\documentclass[a4paper,12pt,reqno]{amsart}

\usepackage{amsmath}
\usepackage{amscd}
\usepackage{amssymb}
\usepackage{mathrsfs}
\usepackage[left=2.5cm,right=2.5cm,bottom=3cm,top=3cm]{geometry}
\newtheorem{thm}{Theorem}[section]

\newtheorem{lem}[thm]{Lemma}

\newtheorem{prop}[thm]{Proposition}
\theoremstyle{definition}

\theoremstyle{remark}
\newtheorem{rem}[thm]{Remark}

\numberwithin{equation}{section}
\begin{document}

\title{On a twisted conical K\"ahler-Ricci flow}

\author{Yashan Zhang}

\address{Beijing International Center for Mathematical Research, Peking University, Beijing 100871, China}
\email{yashanzh@pku.edu.cn}
%\thanks{}

\begin{abstract}
In this paper, we discuss diameter bound and Gromov-Hausdorff convergence of a twisted conical K\"ahler-Ricci flow on the total spaces of some holomorphic submersions. We also observe that, starting from a model conical K\"ahler metric with possibly unbounded scalar curvature, the conical K\"ahler-Ricci flow will instantly have bounded scalar curvature for $t>0$, and the bound is of the form $\frac{C}{t}$. Several key results will be obtained by direct arguments on the conical equation without passing to a smooth approximation. In the last section, we present several remarks on a twisted K\"ahler-Ricci flow and its convergence.
%\keywords{Conical K\"ahler metric \and K\"ahler-Ricci flow \and Gromov-Hausdorff convergence}
%\subclass{MSC 53C44 %\and MSC code2 \and more}
\end{abstract}

\maketitle

\section{Introduction}
The conical K\"ahler-Ricci flow is the K\"ahler-Ricci flow with certain cone singularities, whose existence, regularity and convergence have been widely studied in the recent years, see e.g. \cite{CW1,CW2,Ed15,Ed17,LZ,LZ2,No,Ta,Wy,Yi,Zy} and references therein.
\par In this paper, we shall discuss some properties of the conical K\"ahler-Ricci flow, focusing on the long time collapsing limits and a phenomenon on scalar curvature.

\par Let's begin with the general setup as follows. Let $X$ be an $n$-dimensional compact K\"ahler manifold with a K\"ahler metric $\omega_0$ and $D=\sum_{i=1}^lD_i$ a simple normal crossing divisor on $X$ with every $D_i$ an irreducible complex hypersurface and $\beta_i\in(0,1)$. Fix a defining section $S_i$ for $D_i$ on $X$ and a Hermitian metric $h_i$ on the line bundle $L_i$ associated to $D_i$. We fix a sufficiently small positive constant $\delta$ such that 
$$\omega_0^*:=\omega_0+\delta\sum_{i=1}^{l}\sqrt{-1}\partial\bar\partial|S_i|^{2\beta_i}_{h_i}$$
is a conical K\"ahler metric on $X$ with cone angle $2\pi\beta_i$ along $D_i$, $i=1,...,l$. We call $\omega_0^*$ a model conical K\"ahler metric (with respect to $D$). Then consider the conical K\"ahler-Ricci flow starting from $\omega_0^*$:
\begin{equation}\label{CKRF.1}
\left\{
\begin{aligned}
\partial_t\omega(t)&=-Ric(\omega(t))+2\pi\sum_{i=1}^l(1-\beta_i)[D_i]\\
\omega(0)&=\omega_0^*,
\end{aligned}
\right.
\end{equation}
By \cite{CW1,CW2,LZ,Sh} (which generalizes \cite{C,Ts,TZo}), the conical K\"ahler-Ricci flow \eqref{CKRF.1} has a solution up to
$$T_{max}:=\{t>0|[\omega_0]+t2\pi(c_1(K_X)+\sum_{i=1}^l(1-\beta_i)[D_i])>0\}.$$
See subsection \ref{reduction} for a more precise definition of ``a solution to the conical K\"ahler-Ricci flow".

\subsection{A general phenomenon on scalar curvature along the conical K\"ahler-Ricci flow}\label{general}
Firstly, we would like to note that, as a smooth K\"ahler metric on the open (non-compact) manifold $X\setminus D$, the scalar curvature of $\omega_0^*$ on $X\setminus D$ may not be a bounded function (see e.g. Remark \ref{rem_unbdd}). However, we observe that, starting from a model conical K\"ahler metric $\omega_0^*$ with possibly unbounded scalar curvature, the conical K\"ahler-Ricci flow will instantly have bounded scalar curvature for $t\in(0,t_0]$ ($t_0$ is a sufficiently small number), and the bound is of the form $\frac{C}{t}$. Precisely, we have
\begin{thm}\label{scal_thm}
Let $\omega(t)_{t\in[0,T_{max})}$ be the solution to \eqref{CKRF.1} and $R(t):=R(\omega(t))$ the scalar curvature of $\omega(t)$ on $X\setminus D$. For any fixed $t_0\in(0,T_{max})$, there exists a constant $C\ge1$ such that for any $t\in(0,t_0]$,
\begin{equation}\label{scal_bound}
\sup_{X\setminus D}|R(t)|\le\frac{C}{t}.
\end{equation}
\end{thm}

If we regard the conical K\"ahler-Ricci flow \eqref{CKRF.1} as a smooth K\"ahler-Ricci flow on the open (non-compact) K\"ahler manifold $X\setminus D$, then it seems our observation in Theorem \ref{scal_thm} fits also in the study of the K\"ahler-Ricci flow starting from a K\"ahler metric with unbounded curvature, where the phenomenon in Theorem \ref{scal_thm} has usually happened even for Riemannian curvature, i.e. once one can solve the K\"ahler-Ricci flow from a K\"ahler metric with unbounded curvature, then Riemannian curvature of solution in some cases has a bound of the form $\frac{C}{t}$ instantly, see e.g. \cite{Si,HT} and references therein for more precise results and discussions. From this viewpoint, it seems interesting to ask: \emph{can we prove the analogue result of Theorem \ref{scal_thm} for Riemannian curvature?} Even more, if this is true, then it provides an effective way to produce a (not necessarily model) conical K\"ahler metric with bounded curvature.\\

\noindent\emph{Remark}
We should mention that, in the special case that $X$ is Fano, a weaker version of Theorem \ref{scal_thm} (i.e. replacing the bound $\frac{C}{t}$ in \eqref{scal_bound} by $\frac{C}{t^2}$) seems contained in arguments in \cite[Proposition 4.1, Lemma 4.6]{LZ} and \cite[Proposition 4.1, Theorem 4.4]{LZ2}.

\subsection{Diameter bound and convergence}\label{diam}
Next, we discuss the diameter bound and convergence of the conical K\"ahler-Ricci flow when $X$ additionally admits a holomorphic submersion over another compact K\"ahler manifold.

\par Let $f:X\to Y$ be a holomorphic submersion between two compact K\"ahler manifolds with $n=dim(X)>dim(Y)=k\ge1$ and connected fibers. Fix an irreducible complex hypersurface $D'$ on  $Y$, a defining section $S'$ for $D'$ on $Y$ and a Hermitian metric $h'$ on the line bundle $L'$ associated to $D'$. 
\par Set $D:=f^*D'$, a smooth irreducible complex hypersurface on $X$. Then $S:=f^*S'$ is a defining section for $D$ and $h:=f^*h'$ is a Hermitian metric on $L$, the line bundle associated to $D$. 
\par For an arbitrary K\"ahler metric $\omega_0$ on $X$ and $\beta\in(0,1)$, we fix a sufficiently small positive constant $\delta$ such that
\begin{equation}\label{ref_metric.1}
\omega_0^*:=\omega_0+\delta\sqrt{-1}\partial\bar\partial|S|^{2\beta}_h
\end{equation}
is a model conical K\"ahler metric on $X$ with cone angle $2\pi\beta$ along $D$ and
\begin{equation}\label{ref_metric}
\chi^*=\chi+\delta\sqrt{-1}\partial\bar\partial|S'|^{2\beta}_{h'}
\end{equation}
is a model conical K\"ahler metric on $Y$ with cone angle $2\pi\beta$ along $D'$.
Then, again we consider the following conical K\"ahler-Ricci flow on $X$ starting from $\omega_0^*$:
\begin{equation}\label{CKRF}
\left\{
\begin{aligned}
\partial_t\omega(t)&=-Ric(\omega(t))+2\pi(1-\beta)[D]\\
\omega(0)&=\omega_0^*,
\end{aligned}
\right.
\end{equation}
which has a solution up to
$$T_{max}:=\{t>0|[\omega_0]+t2\pi(c_1(K_X)+(1-\beta)[D])>0\}.$$
We now assume there exists a K\"ahler metric $\chi$ on $Y$ such that
\begin{equation}\label{semiample}
f^*\chi\in\frac{1}{T_{max}}[\omega_0]+2\pi(c_1(K_X)+(1-\beta)[D]).
\end{equation}
Easily, \eqref{semiample} implies
\begin{equation}
f^*(\chi-(1-\beta)R_{h'})\in \frac{1}{T_{max}}[\omega_0]+2\pi c_1(K_X).
\end{equation}

In this paper, given above setting, we study the convergence of the following twisted conical K\"ahler-Ricci flow:
\begin{equation}\label{TCKRF}
\left\{
\begin{aligned}
\partial_t\omega(t)&=-Ric(\omega(t))-\omega(t)+\frac{1}{T_{max}}\omega_0+2\pi(1-\beta)[D]\\
\omega(0)&=\omega_0^*,
\end{aligned}
\right.
\end{equation}

In general, $0<T_{max}\le\infty$. If $T_{max}=\infty$, then \eqref{TCKRF} is just the usual (normalized) conical K\"ahler-Ricci flow and the fiber is Calabi-Yau; if $T_{max}<\infty$, then \eqref{TCKRF} is really a twisted conical K\"ahler-Ricci flow and the fiber is Fano. We will study these two different cases unifiedly. In both cases, \eqref{TCKRF} can be solved for $t\in[0,\infty)$. 
\par One of the most natural problems is to understand the long time behavior of the flow. Our main result can be stated as follows.

\begin{thm}\label{main_thm}
Assume \eqref{semiample} holds. Let $\omega(t)_{t\in[0,\infty)}$ be the solution to \eqref{TCKRF} and $(X,d_t)$ the metric completion of $(X\setminus D,\omega(t))$. We have the following conclusions:
\begin{itemize}
\item[(1)] There exists a positive constant $C<\infty$ such that 
\begin{equation}
\sup_{t\in[0,\infty)}diam(X,d_t)\le C.
\end{equation}

\item[(2)] There exists a conical K\"ahler metric $\overline\chi$ on $Y$ with cone angle $2\pi\beta$ along $D'$ such that, as $t\to\infty$,
\begin{itemize}
\item[(2.1)] $\omega(t)\to f^*\overline\chi$ as currents on $X$;
\item[(2.2)] there exists an $\varepsilon_0\in(0,1)$ with the following property: for any $K\subset\subset X\setminus D$ there exista a constant $C_K\ge1$ such that for any $t\in[1,\infty)$, 
$$|\omega(t)-f^*\overline\chi|_{C^0(K,\omega_0)}\le C_Ke^{-\varepsilon_0 t};$$
\item[(2.3)] if additionally $dim(Y)=1$, then $(X,d_t)\to(Y,\overline d)$ in Gromov-Hausdorff topology, here $(Y,\overline d)$ is the metric completion of $(Y\setminus D',\overline\chi)$.
\end{itemize} 
\end{itemize}
\end{thm}

In our Theorem \ref{main_thm}, we have focused on the volume collapsing case, i.e. $dim(X)>dim(Y)$. For the volume noncollapsing case, we have convergence results on the conical K\"ahler-Ricci flow of Chen-Wang \cite{CW2} when twisted canonical line bundle $K_X+(1-\beta)L$ is trivial or ample and Shen \cite{Sh} when $K_X+(1-\beta)L$ is nef and big (also see Liu-X. Zhang \cite{LZ} on the conical K\"ahler-Ricci flow on Fano manifolds). Items (2.1) and (2.2) in Theorem \ref{main_thm} can be seen as generalizations of Chen-Wang \cite{CW2} and Shen \cite{Sh} to the volume collapsing case. Of course, Theorem \ref{main_thm} are also generalizations of results on the K\"ahler-Ricci flow \cite{ST07,ST12,FZ,TWY} to the conical setting.

\par In the case $T_{max}=\infty$ in \eqref{TCKRF}, the conical K\"ahler metric $\overline\chi$ on $Y$ given in Theorem \ref{main_thm} (2) will satisfy a generalized conical K\"ahler-Einstein equation (see Remark \ref{ke}).
\par In the case $\beta=1$ (i.e. we remove the conical singularity) and $T_{max}<\infty$, the equation \eqref{TCKRF} is essentially introduced (in a slightly different form) in \cite{Zo1} and that corresponding convergence result is even new for this smooth setting, solving some special case of \cite[Conjecture 3.1]{Zo1}.
\par We remark that items (2.1) and (2.2) of Theorem \ref{main_thm} also hold when $D'$ (and hence $D$) is a simple normal crossing divisor as in subsection \ref{general}. However, our arguments for item (1) in Theorem \ref{main_thm} involve the upper bound for bisectional curvature of the model conical metric $\omega_0^*$, which seems only known when $D'$ is of only one irreducible component; we can extend Theorem \ref{main_thm} (1) to simple normal crossing divisor case, provided an upper bound for bisectional curvature for the simple normal crossing case is available.\\

\par Let's also take a look at Theorem \ref{main_thm} (2.3). If $dim(Y)=1$, we know the hypersurface $D'$ on $Y$ is just a single point, say $o\in Y$, and $D=f^*D'$ is a smooth fiber $X_o$ on $X$. Then Theorem \ref{main_thm} (2.3) means that the twisted conical K\"ahler-Ricci flow \eqref{TCKRF} collapses not only the gereric fiber $X_y:=f^{-1}(y)$ for $y\in Y\setminus\{o\}$, but also the cone divisor $X_o$ to a point on the base $Y$, in Gromov-Hausdorff topology. We hope that the additional condition $dim(Y)=1$ can be removed in the future work.

\par We point out that the first result that the conical K\"ahler-Ricci flow will collapse the cone divisor to a point in Gromov-Hausdorff topology was recently obtained by Edwards in \cite{Ed17} on Hirzebruch surfaces (in which Edwards also proved that, in some finite-time noncollapsing case, the conical K\"ahler-Ricci flow can contract cone divisor to a point in Gromov-Hausdorff topology), under certain symmetry condition on the initial model conical metric. Our Theorem \ref{main_thm} (2.3) provides another result on such phenomenon. Note that our result doesn't need any symmetry condition.

\subsection{Organization of this paper} \emph{In the remaining part of this paper, we will always be in the setup introduced in subsection \ref{diam}}. As we will pointed out later (see subsection \ref{scal.1}), a proof for Theorem \ref{scal_thm} is essentially contained in the discussions for setup introduced in subsection \ref{diam}. A more precise organization can be found as follows:
\begin{itemize}
\item In Section \ref{weak_conv}, by using direct arguments on the conical equations, we prove item (2.1) in Theorem \ref{main_thm}. In fact, we will show in Proposition \ref{prop_conv} that the convergence takes place exponentially fast at the level of K\"ahler potentials.
\item In Section \ref{equiv}, by using direct arguments on the conical equations, we prove in Proposition \ref{prop2} that the solution is uniformly equivalent to a family of collapsing conical K\"ahler metric, which will immediately imply item (1) in Theorem \ref{main_thm}. In the proof of Proposition \ref{prop2}, a key step is carried out in Lemma \ref{C2.1}, which seems can \emph{not} be proved by passing to a smooth approximation, see Remark \ref{rem_key} for more comments.
\item In Section \ref{sect_scal}, we show in Lemma \ref{scal} that the twisted scalar curvature along the twisted conical K\"ahler-Ricci flow \eqref{TCKRF} with the assumption \eqref{semiample} has a bound of the from $\frac{C}{\min\{t,1\}}$. Compare with the uniform scalar bound obtained in \cite{Ed15}, our arguments don't rely on some nontrivial properties of a smooth approximation for the initial model K\"ahler metric (see Remark \ref{rem_reasons} for more comments). We point out in subsection \ref{scal.1} that arguments for Lemma \ref{scal} can be applied to prove Theorem \ref{scal_thm}. The arguments in this Section  \ref{sect_scal} need to use a smooth approximation for the twisted conical K\"ahler-Ricci flow. 
\item In Section \ref{sect_C0}, we use direct arguments and the strategy in \cite{TWY} to prove item (2.2) in Theorem \ref{main_thm}. 
\item In Section \ref{sect_gh}, we prove Gromov-Hausdorff convergence stated in item (2.3) in Theorem \ref{main_thm}, in which we will make use of Proposition \ref{prop2} crucially.
\item In Section \ref{sect_rem}, we present several remarks on a twisted K\"ahler-Ricci flow and its convergence.
\end{itemize}

\section{The weak convergence}\label{weak_conv}
In this section, we shall give a proof for item (2.1) in Theorem \ref{main_thm}. 
\subsection{Limiting metric $\overline\chi$}\label{subs_limit}
We first need to construct the limiting metric $\overline\chi$ on $Y$. The construction is essentially the same as in \cite{ST07,ST12} and its generalization in \cite{ZyZz2}. Let's recall some details. We first fix a smooth real function $\rho$ on $X$ such that 
\begin{equation}\label{SRF}
\overline\omega_0:=\omega_0+\sqrt{-1}\partial\bar\partial\rho
\end{equation}
on $X$ is a closed real $(1,1)$-form and  $\overline\omega_0|_{X_y}$ is the unique K\"ahler metric on $X_y$ in class $[\omega_0|_{X_y}]$ with $Ric(\overline\omega_0|_{X_y})=\frac{1}{T_{max}}\omega_0|_{X_y}$ for every $y\in Y$. Secondly we fix a smooth positive volume form $\Omega$ on $X$ with $\sqrt{-1}\partial\bar\partial\log\Omega=f^*\chi-\frac{1}{T_{max}}\omega_0-(1-\beta)2\pi f^*R_{h'}$.
Then define a positive smooth function $G:=\frac{\Omega}{C^{n}_{k}\overline\omega_0^{n-k}\wedge f^*\chi^k}$ on $X$, which descends to a smooth function on $Y$. Set $\overline\chi:=\chi+\sqrt{-1}\partial\bar\partial\psi$ be the unique solution to the following equation:

\begin{equation}\label{limit}
(\chi+\sqrt{-1}\partial\partial\psi)^k=\frac{e^\psi G}{|S'|_{h'}^{2(1-\beta)}}\chi^k.
\end{equation}

We know $\psi\in L^\infty(Y)\cap C^\infty(Y\setminus D')$ \cite{Y,Ko98} and $\overline\chi$ is a conical K\"ahler metric on $Y$ with cone angle $2\pi\beta$ along $D'$, see \cite{CGP,GP}.
We will see $\overline\chi$ is the limit of the twisted conical K\"ahler-Ricci flow \eqref{TCKRF}. 

\begin{rem}\label{ke}
If $T_{max}=\infty$, then $\overline\chi$ constructed above satisfies a generalized conical K\"ahler-Einstein equation (see \cite{ST07}):
$$Ric(\overline\chi)=-\overline\chi+\omega_{WP}+2\pi(1-\beta)[D],$$ 
where $\omega_{WP}$ is the Weil-Petersson form on $Y$.
\end{rem}

\subsection{Reduction to a parabolic complex Monge-Amp\`ere equation}\label{reduction}
Set $\omega_t:=e^{-t}\omega_0+(1-e^{-t})f^*\chi$ and let $\omega(t)=\omega_t+\delta\sqrt{-1}\partial\bar\partial|S|^{2\beta}_h+\sqrt{-1}\partial\bar\partial\varphi(t)$. It is well-known that \eqref{TCKRF} can be reduced to the following 

\begin{equation}\label{TCCMA}
\left\{
\begin{aligned}
\partial_t\varphi(t)&=\log\frac{e^{(n-k)t}(\omega_t+\delta\sqrt{-1}\partial\bar\partial|S|^{2\beta}_h+\sqrt{-1}\partial\bar\partial\varphi(t))^n}{|S|^{-2(1-\beta)}_h\Omega}-\varphi(t)-\delta|S|^{2\beta}_h\\
\varphi(0)&=0,
\end{aligned}
\right.
\end{equation}
With the assumption \eqref{semiample}, by \cite{CW1,CW2,LZ,Sh,Wy} we know \eqref{TCCMA} admits a unique solution $\varphi(t)$ for $t\in[0,\infty)$, where $\varphi(t)\in L^\infty(X)\cap C^\infty((X\setminus D)\times[0,\infty))$ (in fact $\varphi(t)$ is H\"older continuous on $X$ with respect to $\omega_0$ for every $t\ge0$) and $\varphi(t)$ satisfies \eqref{TCCMA} in current sense on $X\times[0,\infty)$ and in smooth sense on $(X\setminus D)\times[0,\infty)$; moreover, for any $T<\infty$ one can find a constant $C\ge1$ such that
\begin{equation}\label{weak_equiv}
C^{-1}\omega_0^*\le\omega(t)\le C\omega_0^*
\end{equation}
holds on $(X\setminus D)\times[0,T]$. 

\par We need some estimates along \eqref{TCKRF} and \eqref{TCCMA}. We point out that most estimates in this paper will be obtained by working with the conical equations \eqref{TCKRF} and \eqref{TCCMA} directly. Such direct arguments have been also used in \cite[Lemmas 2.1, 2.2]{LZ3}. Therefore, we will also provide direct arguments for some existing estimates (see in particular \cite{Ed15}) which have been carried out by passing to a smooth approximation.

\subsection{Uniform bound for $\varphi$ and $\partial_t\varphi$}
We begin with the $C^0$-estimate. In the remaining part of this paper, we assume without loss of any generality that 
\begin{equation}\label{ineq0}
|S'|^2_{h'}\le1.
\end{equation}
Moreover, we fix a small positive constant $\lambda_0\in(0,1]$ such that for any $\lambda\in[-\lambda_0,\lambda_0]$ there holds on $Y\setminus D'$ that
\begin{equation}\label{ineq1}
\frac{1}{2}\chi^*\le\chi^*+\lambda R_{h'}\le2\chi^*,
\end{equation}

\begin{equation}\label{ineq1.1}
\frac{1}{2}\chi\le\chi+\lambda R_{h'}\le2\chi,
\end{equation}
and on $X$ that
\begin{equation}\label{ineq2.1}
\frac{1}{2}\omega_0\le\omega_0+\lambda f^*R_{h'}\le2\omega_0.
\end{equation}

We also fix a sufficiently large constant $T_0\ge1$ such that for any $t\ge T_0$ there holds on $Y\setminus D'$ that
\begin{equation}\label{ineq3}
\frac{1}{2}\chi^*-e^{-t}\chi\ge\frac{1}{3}\chi^*.
\end{equation}

\begin{lem}\label{C0}\cite{Ed15}
There exists a constant $C\ge1$ such that 
\begin{equation}
\sup_{(X\setminus D)\times[0,\infty)}|\varphi|\le C.
\end{equation}
\end{lem}
\begin{proof}
This result has been obtained in \cite{Ed15} by using a smooth approximation. We now provide a direct argument without passing to a smooth approximation. Our arguments follow closely the arguments for smooth K\"ahler-Ricci flow in \cite{ST07,FZ}.

We now use an auxiliary function as in \cite[Lemmas 2.1, 2.2]{LZ3}: for any given $\lambda\in(0,\lambda_0]$, set $\varphi_\lambda:=\varphi+\lambda\log|S|^{2}_h$. We have
\begin{equation}\label{C0.1}
\partial_t{\varphi_\lambda}=\log\frac{e^{(n-k)t}(\omega_t+\delta\sqrt{-1}\partial\bar\partial|S|^{2\beta}_h+\lambda f^*R_{h'}+\sqrt{-1}\partial\bar\partial\varphi_\lambda)^n}{|S|^{-2(1-\beta)}_h\Omega}-\varphi-\delta |S|^{2\beta}_h
\end{equation}
on $X\setminus D$.

For any $T<\infty$, we know $\varphi$ is uniformly bounded on $[0,T]$ (see \cite{Sh,Wy}). Therefore, the maximal value of $\varphi_\lambda$ on $X\times[0,T]$ must achieved at some $(x_T,t_T)$ with $x_T\in X\setminus D$. If $t_T=0$, then $\varphi_\lambda(x_T,t_T)$ is bounded from above by a constant independent of $T$; otherwise $t_T>0$, we can apply maximum principle to \eqref{C0.1} at $(x_T,t_T)$ to see that
\begin{align}\label{ineq3.5}
\varphi(x_T,t_T)\le\log\frac{e^{(n-k)t}(\omega_t+\delta\sqrt{-1}\partial\bar\partial|S|^{2\beta}_h+\lambda f^*R_{h'})^n}{|S|^{-2(1-\beta)}_h\Omega}(x_T,t_T).
\end{align}
Using \eqref{ineq1} gives
\begin{align}
\omega_t+\delta\sqrt{-1}\partial\bar\partial|S|^{2\beta}_h+\lambda f^*R_{h'}&=e^{-t}\omega_0+f^*(\chi^*+\lambda R_{h'})-e^{-t}f^*\chi\nonumber\\
&\le e^{-t}\omega_0+2f^*\chi^*.
\end{align}
and hence
\begin{align}\label{ineq4}
&(\omega_t+\delta\sqrt{-1}\partial\bar\partial|S|^{2\beta}_h+\lambda f^*R_{h'})^n\nonumber\\
&\le(e^{-t}\omega_0+2f^*\chi^*)^n\nonumber\\
&\le C'|S|^{-2(1-\beta)}_h(e^{-(n-k)t}+e^{-(n-k+1)t}+\ldots+e^{-nt})
\end{align}
for some uniform constant $C'\ge1$. Plugging \eqref{ineq4} into \eqref{ineq3.5} implies
$$\varphi(x_T,t_T)\le C''$$
for some uniform constant $C''$ independent on $T$ and $\lambda$. In conclusion, we can choose a constant $C\ge1$ independent on $T\in(0,\infty)$ and $\lambda\in(0,\lambda_0]$ such that
\begin{equation}\label{ineq5}
\sup_{(X\setminus D)\times [0,T]} \varphi_\lambda\le C.
\end{equation}
Let $T\to\infty$ and $\lambda\to0$ in \eqref{ineq5}, we get the desired upper bound for $\varphi$:
\begin{equation}
\sup_{(X\setminus D)\times [0,\infty)} \varphi\le C.
\end{equation}
Now let's look at the lower bound. For any $\lambda\in(0,\lambda]$, set $\varphi^\lambda:=\varphi-\lambda\log|S|^2_h$. Similarly, we have

\begin{equation}\label{C0.2}
\partial_t{\varphi^\lambda}=\log\frac{e^{(n-k)t}(\omega_t+\delta\sqrt{-1}\partial\bar\partial|S|^{2\beta}_h-\lambda f^*R_{h'}+\sqrt{-1}\partial\bar\partial\varphi^\lambda)^n}{|S|^{-2(1-\beta)}_h\Omega}-\varphi-\delta |S|^{2\beta}_h
\end{equation}
on $X\setminus D$.

For any $T<\infty$, we know the minimal value of $\varphi^\lambda$ on $X\times[0,T]$ must be achieved at some $(x_T',t_T')$ with $x_T'\in X\setminus D$. If $t_T'\in[0,T_0]$, then $\varphi_\lambda(x_T',t_T')$ is bounded from below by a constant independent on $T$; otherwise $t_T'>T_0$, we can apply maximum principle to \eqref{C0.2} at $(x_T',t_T')$ to see that
\begin{align}\label{ineq3.6}
\varphi(x_T',t_T')\ge\log\frac{e^{(n-k)t}(\omega_t+\delta\sqrt{-1}\partial\bar\partial|S|^{2\beta}_h-\lambda f^*R_{h'})^n}{|S|^{-2(1-\beta)}_h\Omega}(x_T',t_T')-1.
\end{align}
Using \eqref{ineq1} and \eqref{ineq3} gives
\begin{align}
\omega_t+\delta\sqrt{-1}\partial\bar\partial|S|^{2\beta}_h-\lambda f^*R_{h'}&=e^{-t}\omega_0+f^*(\chi^*-\lambda R_{h'})-e^{-t}f^*\chi\nonumber\\
&\ge e^{-t}\omega_0+\frac{1}{2}f^*\chi^*-e^{-t}f^*\chi\nonumber\\
&\ge e^{-t}\omega_0+\frac{1}{3}f^*\chi^*.
\end{align}
and hence

\begin{align}
(\omega_t+\delta\sqrt{-1}\partial\bar\partial|S|^{2\beta}_h-\lambda f^*R_{h'})^n\ge C^{-1}|S|^{-2(1-\beta)}_he^{-(n-k)t}\Omega
\end{align}
for some uniform constant $C\ge1$. Therefore, we find

\begin{align}
\varphi^\lambda(x_T',t_T')\ge-C
\end{align}
for some constant $C\ge1$ independent on $T\in[0,\infty)$ and $\lambda\in(0,\lambda_0]$. Then as before, by letting $T\to\infty$ and $\lambda\to0$ we conclude

\begin{equation}
\inf_{(X\setminus D)\times[0,\infty)}\varphi\ge-C
\end{equation}
foe some constant $C\ge1$.
\par Lemma \ref{C0} is proved.
\end{proof}

\begin{lem}\label{C00}\cite{Ed15}
There exists a constant $C\ge1$ such that 
\begin{equation}
\sup_{(X\setminus D)\times[0,\infty)}|\partial_t\varphi|\le C.
\end{equation}
\end{lem}

\begin{proof}
This result has been obtained in \cite{Ed15} by using a smooth approximation. Again we present a direct proof without passing to a smooth approximation. By computation in \cite{ST07,FZ}, there holds on $(X\setminus D)\times[0,\infty)$ that
\begin{equation}
(\partial_t-\Delta_{\omega(t)})((e^t-1)\partial_t\varphi-\varphi-\delta|S|^{2\beta}_h)=(n-k)e^t+k-tr_{\omega(t)}\omega_0,
\end{equation}
and so 
\begin{equation}
(\partial_t-\Delta_{\omega(t)})((e^t-1)\partial_t\varphi-\varphi-\delta|S|^{2\beta}_h-(n-k)e^t-kt)=-tr_{\omega(t)}\omega_0.
\end{equation}
As in the last lemma, we now use an auxiliary function similar to \cite[Lemmas 2.1, 2.2]{LZ3}: for any $\lambda\in(0,\lambda_0]$, set $H_\lambda:=(e^t-1)\partial_t\varphi-\varphi-\delta|S|^{2\beta}_h-(n-k)e^t-kt+\lambda\log|S|^2_h$. Then by \eqref{ineq2.1} we have, on $X\setminus D$,

\begin{align}
(\partial_t-\Delta_{\omega(t)})H_\lambda&=-tr_{\omega(t)}(\omega_0-\lambda f^*R_{h'})\nonumber\\
&\le-\frac{1}{2}tr_{\omega(t)}\omega_0<0.
\end{align}

Obviously, for any $T\in(0,\infty)$, since $\partial_t\varphi$ is bounded on $(X\setminus D)\times[0,T]$, the maximal value of $H_\lambda$ on $X\times[0,T]$ must be achieved at some point $(\tilde x_T,0)$ with $\tilde x_T\in X\setminus D$ and so $\sup_{(X\setminus D)\times[0,T]}H_\lambda$ is bounded from above by a positive constant independent on $T\in(0,\infty)$ and $\lambda\in(0,\lambda_0]$. Therefore, we have a positive constant $C\ge1$ such that, on $(X\setminus D)\times[0,\infty)$,

\begin{equation}
(e^t-1)\partial_t\varphi-\varphi-\delta|S|^{2\beta}_h-(n-k)e^t-kt\le C,
\end{equation}
which, combining with the facts that $\partial_t\varphi$ is uniformly bounded on $(X\setminus D)\times[0,1]$ and $\varphi$ is uniformly bounded by Lemma \ref{C0}, implies that $\partial_t\varphi$ is uniformly bounded from above on $(X\setminus D)\times[0,\infty)$.

\par For the lower bound, by a computation similar to \cite{ST07} we have
\begin{equation}
(\partial_t-\Delta_{\omega(t)})(\partial_t\varphi+2\varphi+\delta|S|^{2\beta}_h)=\partial_t\varphi+tr_{\omega(t)}(\omega_t+\delta\sqrt{-1}\partial\bar\partial|S|^{2\beta}_h+f^*\chi)-n-k,   
\end{equation}
and so, for any $\lambda\in(0,\lambda_0]$,
\begin{align}\label{ineq7}
&(\partial_t-\Delta_{\omega(t)})(\partial_t\varphi+2\varphi+\delta|S|^{2\beta}_h-\lambda\log|S|^2_h)\nonumber\\
&=\partial_t\varphi+tr_{\omega(t)}(\omega_t+\delta\sqrt{-1}\partial\bar\partial|S|^{2\beta}_h+f^*\chi-\lambda f^*R_{h'})-n-k,   
\end{align}
on $(X\setminus D)\times[0,\infty)$. By \eqref{ineq1} we have

\begin{align}
\omega_t+\delta\sqrt{-1}\partial\bar\partial|S|^{2\beta}_h+f^*\chi-\lambda f^*R_{h'}&=e^{-t}\omega_0+f^*\chi^*+(1-e^{-t})f^*\chi-\lambda f^*R_{h'}\nonumber\\
&\ge e^{-t}\omega_0+f^*(\chi*-\lambda R_{h'})\nonumber\\
&\ge e^{-t}\omega_0+\frac{1}{2}f^*\chi^*,
\end{align}
which implies
\begin{align}\label{ineq8}
tr_{\omega(t)}(\omega_t+\delta\sqrt{-1}\partial\bar\partial|S|^{2\beta}_h+f^*\chi-\lambda f^*R_{h'})&\ge tr_{\omega(t)}(e^{-t}\omega_0+\frac{1}{2}f^*\chi^*)\nonumber\\
&\ge n\left(\frac{(e^{-t}\omega_0+\frac{1}{2}f^*\chi^*)^n}{\omega(t)^n}\right)^{\frac{1}{n}}\nonumber\\
&\ge n\left(\frac{C^{-1}e^{-(n-k)t}|S|^{-2(1-\beta)}_h\Omega}{e^{-(n-k)t}|S|^{-2(1-\beta)}_he^{\partial_t\varphi
+\varphi+\delta|S|^{2\beta}_h}\Omega}\right)^{\frac{1}{n}}\nonumber\\
&\ge C^{-1}e^{-\frac{\partial_t\varphi}{n}},
\end{align}
for some uniform positive constant $C$ independent on $\lambda\in(0,\lambda_0], $where we have used $\varphi$ is uniformly bounded by Lemma \ref{C0}. Plugging \eqref{ineq8} into \eqref{ineq7} gives

\begin{equation}
(\partial_t-\Delta_{\omega(t)})(\partial_t\varphi+2\varphi+\delta|S|^{2\beta}_h-\lambda\log|S|^2_h)\ge\partial_t\varphi+C^{-1}e^{-\frac{\partial_t\varphi}{n}}-n-k
\end{equation}
on $(X\setminus D)\times[0,\infty)$. Now by maximum principle we can find a constant $C\ge1$ such that for any $\lambda\in(0,\lambda_0]$
$$\inf_{(X\setminus D)\times[0,\infty)}\partial_t\varphi+2\varphi+\delta|S|^{2\beta}_h-\lambda\log|S|^2_h\ge-C.$$
Leting $\lambda\to0$ gives the desired uniform lower bound for $\partial_t\varphi$ on $(X\setminus D)\times[0,\infty)$.
\par Lemma \ref{C00} is proved.
\end{proof}

\subsection{Convergence of K\"ahler potentials} 
We come to the main result of this section.

\begin{prop}\label{prop_conv}
There exists a constant $C\ge1$ such that for any $t\in[0,\infty)$,
\begin{equation}\label{conv_00}
\sup_{X\setminus D}|\varphi(t)+\delta|S|^{2\beta}_h-f^*\psi|\le Ce^{-\frac{3}{4}t}.
\end{equation}
\end{prop}

\begin{proof}
Set $V:=\varphi+\delta|S|^{2\beta}_h-f^*\psi-e^{-t}\rho$, where $\psi$ is the unique solution to \eqref{limit} and $\rho$ is the function in \eqref{SRF}. By \eqref{limit} and \eqref{TCCMA} we have (see \cite{ST07})
\begin{equation}\label{ineq9}
\partial_t V=\log\frac{e^{(n-k)t}(e^{-t}\overline\omega_0+f^*\overline\chi-e^{-t}f^*\chi+\sqrt{-1}\partial\bar\partial V)^n}{C_k^n\overline\omega_0^{n-k}\wedge f^*\overline\chi^k}-V
\end{equation}
on $(X\setminus D)\times[0,\infty)$. For any $\lambda\in(0,\lambda_0]$ we set $V_\lambda:=V+e^{-t}\lambda\log|S|^2_h$. By \eqref{ineq9} one easily has

\begin{equation}\label{ineq10}
\partial_t V_\lambda=\log\frac{e^{(n-k)t}(e^{-t}\overline\omega_0+f^*\overline\chi-e^{-t}f^*(\chi+\lambda R_{h'})+\sqrt{-1}\partial\bar\partial V_\lambda)^n}{C_k^n\overline\omega_0^{n-k}\wedge f^*\overline\chi^k}-V_\lambda
\end{equation}
on $(X\setminus D)\times[0,\infty)$. For any $t\in[0,\infty)$, the maximal value $V_{\lambda,max}(t)$ of $V_\lambda(t)$ must be achieved at some point in $X\setminus D$. Hence, applying the maximum principle in \eqref{ineq10} and then using \eqref{ineq1.1} give

\begin{align}
\partial_tV_{\lambda,max}+V_{\lambda,max}&\le\log\frac{e^{(n-k)t}(e^{-t}\overline\omega_0+f^*\overline\chi-e^{-t}f^*(\chi+\lambda R_{h'}))^n}{C_k^n\overline\omega_0^{n-k}\wedge f^*\overline\chi^k}\nonumber\\
&\le\log\frac{e^{(n-k)t}(e^{-t}\overline\omega_0+f^*\overline\chi)^n}{C_k^n\overline\omega_0^{n-k}\wedge f^*\overline\chi^k}\nonumber\\
&\le \log(1+Ce^{-t})\nonumber\\
&\le Ce^{-t},
\end{align}
i.e.,

\begin{equation}
\partial_t(e^tV_{\lambda,max}-Ct)\le0,
\end{equation}
and so

$$V_{\lambda,max}\le C'e^{-t}(t+1)\le Ce^{-\frac{3}{4}t}$$
for some uniform constant $C\ge1$ independent on $\lambda\in(0,\lambda_0]$.  Let $\lambda\to0$, we find a constant $C\ge1$ such that

\begin{equation}\label{upper1}
\sup_{(X\setminus D)\times[0,\infty)}V\le Ce^{-\frac{3}{4}t}.
\end{equation}

For the lower bound we define $V^\lambda:=V-e^{-t}\lambda\log|S|^2_h$ for any $\lambda\in(0,\lambda_0]$. We know the minimal value $V^\lambda_{min}(t)$ of $V^\lambda(t)$ must be achieved at some point in $X\setminus D$. Similarly, we have

\begin{align}
\partial_tV^{\lambda}_{min}+V^{\lambda}_{min}&\ge\log\frac{e^{(n-k)t}(e^{-t}\overline\omega_0+f^*\overline\chi-e^{-t}f^*(\chi-\lambda R_{h'}))^n}{C_k^n\overline\omega_0^{n-k}\wedge f^*\overline\chi^k}\nonumber\\
\end{align}
We now fix a large $T_1$ such that $e^{-t}\overline\omega_0+f^*\overline\chi-2e^{-t}f^*\chi>0$ for any $t\ge T_1$ (in particular, $T_1$ does not depend on $\lambda\in(0,\lambda_0]$). Then, for any $t\ge T_1$,
\begin{align}
\log\frac{e^{(n-k)t}(e^{-t}\overline\omega_0+f^*\overline\chi-e^{-t}f^*(\chi-\lambda R_{h'}))^n}{C_k^n\overline\omega_0^{n-k}\wedge f^*\overline\chi^k}&\ge\log\frac{e^{(n-k)t}(e^{-t}\overline\omega_0+f^*\overline\chi-2e^{-t}f^*\chi)^n}{C_k^n\overline\omega_0^{n-k}\wedge f^*\overline\chi^k}\nonumber\\
&\ge \log(1-Ce^{-t})\nonumber\\
&\ge-Ce^{-t},
\end{align}
where $C\ge1$ is a uniform positive constant independent on $t\in[0,\infty)$ and $\lambda\in(0,\lambda_0]$, and we may increase $T_1$ to make sure that $1-Ce^{-t}>0$ for any $t\ge T_1$.
Therefore, we arrive at 
\begin{equation}
\partial_t(e^tV^\lambda_{min}+Ct)\ge0,
\end{equation}
on $t\in[T_1,\infty)$, from which we conclude
$$\inf_{(X\setminus D)\times[T_1,\infty)}V^\lambda\ge-Ce^{-\frac{3}{4}t}$$
for some constant $C\ge1$ independent on $\lambda\in(0,\lambda_0]$. By letting $\lambda\to0$, we find
\begin{equation}\label{lower.1}
\inf_{(X\setminus D)\times[T_1,\infty)}V\ge-Ce^{-\frac{3}{4}t}.
\end{equation}
Moreover, since $V$ is uniformly bounded on $(X\setminus D)\times[0,T_1]$, after possibly increasing $C$ in \eqref{lower.1} we have
\begin{equation}\label{lower.2}
\inf_{(X\setminus D)\times[0,\infty)}V\ge-Ce^{-\frac{3}{4}t}.
\end{equation}

By combining \eqref{upper1}, \eqref{lower.2} and the fact that $\rho$ is a bounded function on $X$ we conclude \eqref{conv_00} and Proposition \ref{prop_conv} is proved.
\end{proof}

\subsection{Proof of Theorem \ref{main_thm} (2.1)}

Consequently, we have

\begin{proof}[Proof of Theorem \ref{main_thm} (2.1)]
By Lemma \ref{C0} and Proposition \ref{prop_conv}, we conclude that, as $t\to\infty$, $\varphi+\delta|S|^{2\beta}_h\to f^*\psi$ in $L^1(X,\Omega)$-topology and so $\omega(t)\to f^*\overline\chi$ as currents on $X$.
\end{proof}

\section{uniform equivalence to a family of collapsing conical K\"ahler metrics}\label{equiv}
Recall \eqref{weak_equiv} that for any $T<\infty$, there exists a constant $A=A(T)\ge1$ such that 
\begin{equation}\label{weak_equiv.1}
A^{-1}\omega_0^*\le\omega(t)\le A\omega_0^*
\end{equation}
holds on $(X\setminus D)\times[0,T]$. In our setting, we know the K\"ahler class $[\omega(t)]=e^{-t}[\omega_0]+(1-e^{-t})[f^*\chi]$ along the twisted conical K\"ahler-Ricci flow \eqref{TCKRF} will converge to $[f^*\chi]$, which lies on the boundary of K\"ahler cone of $X$ and is not strictly positive. Therefore, $A=A(T)$ in \eqref{weak_equiv.1} should go to $\infty$ as $T\to\infty$, i.e. $\omega(t)$ should not be uniformly equivalent to a fixed conical K\"ahler metric on $(X\setminus D)\times[0,\infty)$. 
\par In this section, we will show that $\omega(t)$ is uniformly equivalent to a family of conical K\"ahler metrics on $X$ whose K\"ahler classes will degenerate to $[f^*\chi]$ (see Proposition \ref{prop2}). Let's begin by the following key lemma, which will play an important role in our later discussions. Recall that, defined by \eqref{ref_metric}, $\chi^*$ is a model conical metric on $Y$ with cone angle $2\pi\beta$ along $D'$.

\begin{lem}[key lemma]\label{C2.1}
There exists a constant $C\ge1$ such that there holds on $(X\setminus D)\times[0,\infty)$ that
\begin{equation}
tr_{\omega(t)}f^*\chi^*\le C.
\end{equation}
\end{lem}

\begin{proof}
The proof makes use of Lemmas \ref{C0}, \ref{C00}, a parabolic Schwarz lemma argument \cite{ST07,Y78} and, especially, an upper bound for bisectional curvature of $\chi^*$ in \cite[Proposition A.1]{JMR}. Firstly, we note that by \eqref{weak_equiv.1} $tr_{\omega(t)}f^*\chi^*$ is uniformly bounded on $(X\setminus D)\times[0,1]$. Secondly, for any given point $x\in X\setminus D$, we choose normal coordinates for $\omega(t)$ around $x$ and $\chi^*$ around $f(x)$, in which we write $tr_{\omega(t)}f^*\chi^*=g^{\bar ji}f^\alpha_i\overline{f^\beta_j}\chi^*_{\alpha\bar\beta}$. A standard computation gives

\begin{align}\label{P2.2}
&(\partial_t-\Delta_{\omega(t)})tr_{\omega(t)}f^*\chi^*\nonumber\\
&=tr_{\omega(t)}f^*\chi^*-\langle\frac{1}{T_{max}}\omega_0,f^*\chi^*\rangle_{\omega(t)}+g^{\bar lk}g^{\bar j i}Rm(\chi^*)_{\alpha\bar\beta\gamma\bar\delta}f^{\alpha}_i\overline{f^\beta_j}f^\gamma_k\overline{f^\delta_l}-g^{\bar lk}g^{\bar ji}(\partial_kf^\alpha_i)(\partial_lf^\beta_j)\chi^*_{\alpha\bar\beta}\nonumber\\
&\le tr_{\omega(t)}f^*\chi^*+B(tr_{\omega(t)}f^*\chi^*)^2-g^{\bar lk}g^{\bar ji}(\partial_kf^\alpha_i)(\partial_lf^\beta_j)\chi^*_{\alpha\bar\beta},
\end{align}
and so 
\begin{align}
(\partial_t-\Delta_{\omega(t)})\log tr_{\omega(t)}f^*\chi^*\le Btr_{\omega(t)}f^*\chi^*+1
\end{align}
on $(X\setminus D)\times[0,\infty)$, where $B$ is an upper bound for bisectional curvature of $\chi^*$ on $Y\setminus D'$ given by \cite[Proposition A.1]{JMR} (also see \cite[Section 7]{Br} and \cite[Lemma 2.3]{JMR} for the case $\beta\in(0,\frac{1}{2})$) and we have used
\begin{equation}\label{P2.3}
|\nabla tr_{\omega(t)} f^*\chi^*|^2-(tr_{\omega(t)} f^*\chi^*)g^{\bar lk}g^{\bar ji}\partial_kf^\alpha_{i}\overline{\partial_lf^\beta_{j}}\chi^*_{\alpha\bar\beta}\le0,
\end{equation}

On the other hand, by Lemma \ref{C00} and \eqref{ineq3} we have a uniform constant $\tilde C\ge1$ such that
\begin{align}
(\partial_t-\Delta_{\omega(t)})\varphi&=\partial_t\varphi-tr_{\omega(t)}(\omega(t)-e^{-t}\omega_0)+tr_{\omega(t)}f^*(\chi^*-e^{-t}\chi)\nonumber\\
&\ge tr_{\omega(t)}f^*(\chi^*-e^{-t}\chi)-\tilde C\nonumber\\
&\ge\frac{5}{6}tr_{\omega(t)}f^*\chi^*-\tilde C
\end{align}
on $(X\setminus D)\times[T_0,\infty)$, here $T_0$ is given in \eqref{ineq3}. Therefore, we arrive at
\begin{equation}
(\partial_t-\Delta_{\omega(t)})(\log tr_{\omega(t)}f^*\chi^*-\frac{6}{5}(B+1)\varphi)\le-tr_{\omega(t)}f^*\chi^*+1+\frac{6}{5}\tilde C(B+1)
\end{equation}
and so, for any $\lambda\in(0,\lambda_0]$,
\begin{align}\label{ineq11}
&(\partial_t-\Delta_{\omega(t)})(\log tr_{\omega(t)}f^*\chi^*-\frac{6}{5}(B+1)\varphi+\lambda\log|S|^2_h)\nonumber\\
&\le-tr_{\omega(t)}f^*(\chi^*+\lambda R_{h'})+1+\frac{6}{5}\tilde C(B+1)\nonumber\\
&\le-\frac{1}{2}tr_{\omega(t)}f^*\chi^*+\frac{6}{5}\tilde C(B+1).
\end{align}
on $(X\setminus D)\times[T_0,\infty)$.  Now by applying the maximum principle in \eqref{ineq11}, Lemma \ref{C0} and \eqref{weak_equiv.1}, we can find a constant $C\ge1$ such that for any $\lambda\in(0,\lambda_0]$ we have
\begin{align}
\log tr_{\omega(t)}f^*\chi^*+\lambda\log|S|^2_h\le C,
\end{align}
which implies 
\begin{equation}\label{upper.2}
tr_{\omega(t)}f^*\chi^*\le C
\end{equation}
on $(X\setminus D)\times[T_0,\infty)$. By \eqref{weak_equiv.1}, after possibly increasing $C$ in \eqref{upper.2}, we have
\begin{equation}\label{upper.3}
\sup_{(X\setminus D)\times[0,\infty)}tr_{\omega(t)}f^*\chi^*\le C.
\end{equation}
Lemma \ref{C2.1} is proved.
\end{proof}

We would like to make a remark on Lemma \ref{C2.1} and its proof.

\begin{rem}\label{rem_key}
The key Lemma \ref{C2.1} seems can \emph{not} be proved by the usual smooth approximation arguments in e.g. \cite{Ed15,LZ}. A key point is that, while $\chi^*$ has an upper bound for bisectional curvature, it seems  unclear how to approximate $\chi^*$ (in a suitable sense) by a family of smooth K\"ahler metrics with bisectional curvature uniformly bounded from above. In our argument, we don't need to pass to a smooth approximation of conical equation and the upper bound for bisectional curvature of $\chi^*$ itself is enough to carry out our Lemma \ref{C2.1}.
\end{rem}

Next we make use of Lemma \ref{C2.1} and an argument in \cite{ST07} to show
\begin{lem}\label{C2.2}
There exists a constant $C\ge1$ such that there holds on $(X\setminus D)\times[0,\infty)$ that
\begin{equation}
tr_{\omega(t)}(e^{-t}\omega_0)\le C.
\end{equation}
\end{lem}

\begin{proof}
As before, by \eqref{weak_equiv.1} we know $tr_{\omega(t)}(e^{-t}\omega_0)$ is uniformly bounded on $(X\setminus D)\times[0,1]$. 
\par Firstly, we have two positive constants $B_0$ (an upper bound for bisectional curvature of $\omega_0$) and $C$ such that
\begin{equation}\label{ineq12}
(\partial_t-\Delta_{\omega(t)})\log tr_{\omega(t)}(e^{-t}\omega_0)\le B_0tr_{\omega(t)}\omega_{0}+C
\end{equation}
on $(X\setminus D)\times[0,\infty)$. 
\par Next, following \cite{ST07}, we define a function $\overline\varphi$ on $(Y\setminus D')\times[0,\infty)$: for any $(y,t)\in(Y\setminus D')\times[0,\infty)$, 
\begin{equation}
\overline\varphi(y,t):=\frac{\int_{X_y}\varphi(t)(\omega_{0}|_{X_y})^{n-k}}{\int_{X_y}(\omega_{0}|_{X_y})^{n-k}}.
\end{equation}
Then we have
\begin{align}
(\omega_0|_{X_y}+\sqrt{-1}\partial\bar\partial(e^t(\varphi|_{X_y}-\overline\varphi(y,t))))^{n-k}&=(e^{t}\omega(t)|_{X_y})^{n-k},
\end{align}
and so
\begin{align}
&\frac{(\omega_0|_{X_y}+\sqrt{-1}\partial\bar\partial(e^t(\varphi|_{X_y}-\overline\varphi)))^{n-k}}{(\omega_0|_{X_y})^{n-k}}\nonumber\\
&=\frac{(e^{t}\omega(t)|_{X_y})^{n-k}}{(\omega_0|_{X_y})^{n-k}}\nonumber\\
&=\frac{e^{(n-k)t}\omega(t)^{n-k}\wedge f^*(\chi^*)^k}{\omega_0^{n-k}\wedge f^*(\chi^*)^k}\nonumber\\
&=\frac{\omega(t)^{n-k}\wedge f^*(\chi^*)^k}{\omega(t)^n}\cdot\frac{e^{(n-k)t}\omega(t)^n}{\omega_0^{n-k}\wedge f^*(\chi^*)^k}\nonumber\\
&=\frac{\omega(t)^{n-k}\wedge f^*(\chi^*)^k}{\omega(t)^n}\cdot\frac{e^{\partial_t\varphi+\varphi+\delta|S|^{2\beta}_h}|S|^{-2(1-\beta)}_h\Omega}{\omega_0^{n-k}\wedge f^*(\chi^*)^k}\nonumber\\
&\le C' (tr_{\omega(t)}f^*\chi^*)^{k}\nonumber\\
&\le C,
\end{align}
where we have used Lemmas \ref{C0}, \ref{C00} and \ref{C2.1}, and $C$ in the last line is a positive constant uniform for $(x,t)\in(X\setminus D)\times[0,\infty)$ (we point out that Lemma \ref{C2.1} plays a crucial role in this step). Then by Yau \cite{Y} we get a constant $C\ge1$ such that for all $(y,t)\in(Y\setminus D')\times[0,\infty)$,
\begin{equation}\label{ineq12.5}
\sup_{X_y}|e^t(\varphi-\overline\varphi)|\le C.
\end{equation}
Note that $|S|^{2\beta}_h=f^*|S'|^{2\beta}_{h'}$ is constant along every fiber $X_y$ and hence 
$$(\varphi+\delta|S|^{2\beta}_h)-(\overline\varphi+\delta|S|^{2\beta}_h)=\varphi-\overline\varphi.$$
Now by using arguments in \cite{ST07} (also see \cite{FZ}), Lemma \ref{C00} and Lemma \ref{C2.1} (in fact a weaker version that $tr_{\omega(t)}f^*\chi\le C$ is enough) we have a constant $C\ge1$ such that 
\begin{align}\label{ineq13}
(\partial_t-\Delta_{\omega(t)})(e^t(\varphi-\overline\varphi))&=(\partial_t-\Delta_{\omega(t)})\left(e^t((\varphi+\delta|S|^{2\beta}_h)-(\overline\varphi+\delta|S|^{2\beta}_h))\right)\nonumber\\
&\ge tr_{\omega(t)}\omega_0-Ce^t
\end{align}
holds on $(X\setminus D)\times[0,\infty)$. Combining \eqref{ineq12} and \ref{ineq13} gives
\begin{equation}
(\partial_t-\Delta_{\omega(t)})(\log tr_{\omega(t)}(e^{-t}\omega_0)-(B_0+1)e^t(\varphi-\overline\varphi))\le -tr_{\omega(t)}\omega_0+Ce^t
\end{equation}
on $(X\setminus D)\times[0,\infty)$, where $C\ge1$ is a uniform constant. Now for any $\lambda\in(0,\lambda_0]$, by \eqref{ineq2.1} there holds on $(X\setminus D)\times[0,\infty)$ that
\begin{equation}\label{ineq14}
(\partial_t-\Delta_{\omega(t)})(\log tr_{\omega(t)}(e^{-t}\omega_0)-(B_0+1)e^t(\varphi-\overline\varphi)+\lambda\log|S|^2_h)\le -\frac{1}{2}tr_{\omega(t)}\omega_0+Ce^t.
\end{equation}
By applying the maximum principle in \eqref{ineq14} and using \eqref{ineq12.5} we find a constant $C\ge1$ such that for any $\lambda\in(0,\lambda_0]$,
$$\sup_{(X\setminus D)\times[0,\infty)}(\log tr_{\omega(t)}(e^{-t}\omega_0)-(B_0+1)e^t(\varphi-\overline\varphi)+\lambda\log|S|^2_h)\le C,$$
from which, by letting $\lambda\to0$ and using \eqref{ineq12.5} again, we conclude a constant $C\ge1$ such that
\begin{equation}
\sup_{(X\setminus D)\times[0,\infty)}tr_{\omega(t)}(e^{-t}\omega_0)\le C.
\end{equation}
Lemma \ref{C2.2} is proved.
\end{proof}

The main result of this section is the following, which says $\omega(t)$ is uniformly equivalent to a family of collapsing conical K\"ahler metrics.

\begin{prop}\label{prop2}
There exists a constant $C\ge1$ such that 
\begin{equation}\label{C2.3}
C^{-1}(e^{-t}\omega_0+f^*\chi^*)\le\omega(t)\le C(e^{-t}\omega_0+f^*\chi^*)
\end{equation}
holds on $(X\setminus D)\times[0,\infty)$.
\end{prop}

\begin{proof}
The left hand side follows from Lemmas \ref{C2.1} and \ref{C2.2} directly. For the right hand side, we look at
\begin{align}
tr_{(e^{-t}\omega_0+f^*\chi^*)}\omega(t)&\le\frac{1}{(n-1)!}(tr_{\omega(t)}(e^{-t}\omega_0+f^*\chi^*))^{n-1}\cdot\frac{\omega(t)^n}{(e^{-t}\omega_0+f^*\chi^*)^n}\nonumber\\
&\le C\frac{e^{-(n-k)t}e^{\partial_t\varphi+\varphi+\delta|S|^{2\beta}_h}|S|^{-2(1-\beta)}_h\Omega}{e^{-(n-k)t}\omega_0^{n-k}\wedge f^*(\chi^*)^k}\nonumber\\
&\le C,
\end{align}
on $(X\setminus D)\times[0,\infty)$, where we have used left hand side of \eqref{C2.3} and upper bound of $\partial_t\varphi,\varphi$ given by Lemmas \ref{C0} and \ref{C00}. Then we conclude that 
$$\omega(t)\le C(e^{-t}\omega_0+f^*\chi^*)$$
on $(X\setminus D)\times[0,\infty)$.
\par Proposition \ref{prop2} is proved.
\end{proof}

\begin{rem}
Before moving to the next step, we would like to provide a more direct argument for a special case of Lemma \ref{C2.2}, i.e. the case that $T_{max}=\infty$, $X=E\times Y$ with $E$ a smooth $(n-k)$-dimensional torus and $f:X=E\times Y\to Y$ is the projection. Denote $\tilde f:X\to E$ be the projection to the factor $E$ and $\omega_E$ a fixed flat K\"ahler metric on $E$. Then we shall show that there exists a constant $C\ge1$ such that for $(x,t)\in (X\setminus D)\times[0,\infty)$,
\begin{equation}\label{e.0}
tr_{\omega(t)}(e^{-t}\tilde f^*\omega_E)\le C.
\end{equation}
In fact, using $\omega_E$ is flat, we easily have
\begin{equation}
(\partial_t-\Delta_{\omega(t)})tr_{\omega(t)}(e^{-t}\tilde f^*\omega_E)\le0\nonumber
\end{equation}
and so
\begin{equation}
(\partial_t-\Delta_{\omega(t)})(tr_{\omega(t)}(e^{-t}\tilde f^*\omega_E)+\lambda\log|S|^2_h)\le \lambda tr_{\omega(t)}R_h\le \lambda \tilde Ctr_{\omega(t)}f^*\chi\le\lambda \tilde C_1\nonumber,
\end{equation}
namely,
\begin{equation}\label{e.1}
(\partial_t-\Delta_{\omega(t)})(tr_{\omega(t)}(e^{-t}\tilde f^*\omega_E)+\lambda\log|S|^2_h-\lambda\tilde C_1t)\le0
\end{equation}
on $(X\setminus D)\times[0,\infty)$. For any $T<\infty$, we can apply the maximal principle in \eqref{e.1} to see 
$$\sup_{(X\setminus D)\times[0,T]}(tr_{\omega(t)}(e^{-t}\tilde f^*\omega_E)+\lambda\log|S|^2_h-\lambda\tilde C_1t)\le C,$$
where $C$ does not depend on $T\in(0,\infty)$ and $\lambda\in(0,\lambda_0]$. Letting $T\to\infty$ and $\lambda\to0$ gives the desired result \eqref{e.0}. Having \eqref{e.0}, we can use arguments in Proposition \ref{prop2} to see that $\omega(t)$ is uniformly equivalent to $e^{-t}\tilde f^*\omega_E+f^*\chi^*$ on $(X\setminus D)\times[0,\infty)$.
\end{rem}

\subsection{Proof of Theorem \ref{main_thm} (1)} 
We now end this section by proving diameter upper bound of $(X,\omega(t))$, i.e. item (1) of Theorem \ref{main_thm}.

\begin{proof}[Proof of Theorem \ref{main_thm} (1)]
By the right hand side of \eqref{C2.3} in Proposition \ref{prop2} we can find a constant $C\ge1$ such that
$$\omega(t)\le C\omega_0^*$$
holds on $(X\setminus D)\times[0,\infty)$, where $\omega_0^*$ is defined in \eqref{ref_metric.1}, the reference conical K\"ahler metric on $X$ with cone angle $2\pi\beta$ along $D$. Let $(X,d_0)$ be the metric completion of $(X,\omega_0^*)$, which is of finite diameter. Therefore, 
$$diam(X,d_t)\le C\cdot diam(X,d_0)$$
is unifromly bounded from above.
\end{proof}

\begin{rem}
Easily, the left hand side of \eqref{C2.3} in Proposition \ref{prop2} also implies a uniform positive lower bound for $diam(X,d_t)$.
\end{rem}

\section{A bound for the twisted scalar curvature}\label{sect_scal}
For later discussions, we need to bound the twisted scalar curvature along the twisted conical K\"ahler-Ricci flow \eqref{TCKRF}. %This is essentially contained in \cite[Theorem 1.1]{Ed15}.

We define the twisted Ricci curvature by
\begin{equation}
\tilde{Ric}(\omega(t)):=Ric(\omega(t))-\frac{1}{T_{max}}\omega_0
\end{equation}
and the twisted scalar curvature by
\begin{equation}
\tilde R(\omega(t)):=tr_{\omega(t)}\tilde{Ric}(\omega(t))=R(\omega(t))-\frac{1}{T_{max}}tr_{\omega(t)}\omega_0.
\end{equation}
Note that if $T_{max}=\infty$, the twisted term vanishes and the twisted curvatures defined above are the usual curvatures.
\par We will need the following result.

\begin{lem}\label{scal}\cite{Ed15}
There exists a constant $C\ge1$ such that for any $t\in(0,\infty)$,
\begin{equation}
\sup_{X\setminus D}|\tilde R|(t)\le \frac{C}{min\{t,1\}}.
\end{equation}
\end{lem} 

In the case $T_{max}=\infty$, the twisted term $\frac{1}{T_{max}}\omega_0$ vanishes and the twisted scalar is just the scalar curvature, which is uniformly bounded by \cite[Theorem 1.1]{Ed15} using a smooth approximation argument (this is a generalization of Song and Tian's result on the K\"ahler-Ricci flow \cite{ST16}) In fact, while our Lemma \ref{scal} only provides uniform bound for $t\ge1$, \cite[Theorem 1.1]{Ed15} says the bound for scalar curvature can be uniform for all $t\in[0,\infty)$ (i.e. $\sup_{(X\setminus D)\times[0,\infty)}|\tilde R|\le C<\infty$), where some properties of a smooth approximation for $\omega_0^*$ is involved (see discussions next to \eqref{lower2}, \eqref{P2.7} and \eqref{P2.10} for more details). Though bounding scalar curvature along the conical K\"ahler-Ricci flow has been studied in some different settings, see e.g. \cite{Ed15,No}, we still would like to provide an argument for the above Lemma \ref{scal}. Let's list in the following Remark \ref{rem_reasons} the main reasons for doing this.
\begin{rem}\label{rem_reasons}
We provide an argument for Lemma \ref{scal}, mainly due to twofold reasons: (1) we will use slightly different quantum to apply the maximum principle (see footnotes next to \eqref{footnote1} and \eqref{P2.6} for more details); (2) we can avoid using some properties of a smooth approximation $\omega_{0,\epsilon}^*$ (given in \eqref{appro}) for $\omega_0^*$, which seems needed in \cite[Theorem 1.1, Sections 4,7,8]{Ed15} and \cite[Theorem A, Sections 6,7]{No}  (see discussions next to \eqref{lower2}, \eqref{P2.7} and \eqref{P2.10} for more details).
\end{rem}

\begin{rem}\label{rem_unbdd}
It seems in general we can not improve the bound $\frac{C}{t}$ to a uniform bound $C$ near $t=0$, as the initial model conical metric $\omega^*_0$ may have unbounded scalar curvature. For example, as pointed out in \cite[Lemma 3.14]{Ru}, if $\beta\in(\frac{1}{2},1)$, the computation in \cite[Proposition A.1]{JMR} provides not only the existence of an upper bound for bisectional curvature, but also the nonexistence of a lower bound for bisectional curvature, and hence the nonexistence of a lower bound for scalar curvature.
\end{rem}

\begin{proof}[Proof of Lemma \ref{scal}]
We need to use a smooth approximation for the conical equations \eqref{TCKRF} and \eqref{TCCMA}, introduced in \cite{LZ,Wy} and also used in e.g. \cite{Sh,Ed15}. Following \cite{CGP} we define
$$\eta_\epsilon:=\beta\int_{0}^{|S|^{2}_h}\frac{(r+\epsilon^2)^\beta-\epsilon^{2\beta}}{r}dr,$$ 
and 
\begin{equation}\label{appro}
\omega^*_{0,\epsilon}:=\omega_0+\sqrt{-1}\partial\bar\partial\eta_\epsilon.
\end{equation}
Then one can approximate \eqref{TCCMA} by
\begin{equation}\label{TCCMA.1}
\left\{
\begin{aligned}
\partial_t\varphi_\epsilon(t)&=\log\frac{e^{(n-k)t}(\omega_t+\delta\sqrt{-1}\partial\bar\partial\eta_\epsilon+\sqrt{-1}\partial\bar\partial\varphi_\epsilon(t))^n}{(|S|^{2}_h+\epsilon^2)^{-(1-\beta)}\Omega}-\varphi_\epsilon(t)-\delta\eta_\epsilon\\
\varphi_\epsilon(0)&=0,
\end{aligned}
\right.
\end{equation}
and, correspondingly, if we let $\omega_\epsilon(t):=\omega_t+\delta\sqrt{-1}\partial\bar\partial\eta_\epsilon+\sqrt{-1}\partial\bar\partial\varphi_\epsilon(t)$, then we have
\begin{equation}\label{TCKRF.1}
\left\{
\begin{aligned}
\partial_t\omega_\epsilon(t)&=-Ric(\omega_\epsilon(t))-\omega_\epsilon(t)+\frac{1}{T_{max}}\omega_0+(1-\beta)(\sqrt{-1}\partial\bar\partial\log(|S|^2_h+\epsilon^2)+R_h)\\
\omega(0)&=\omega^*_{0,\epsilon},
\end{aligned}
\right.
\end{equation}
Set $A_\epsilon:=(1-\beta)(\sqrt{-1}\partial\bar\partial\log(|S|^2_h+\epsilon^2)+R_h)$, then it suffices to get a bound for $\tilde R_\epsilon:=tr_{\omega_\epsilon(t)}(Ric(\omega_\epsilon(t))-\frac{1}{T_{max}}\omega_0-A_\epsilon)$. Let's first look at the easier part, i.e. lower bound.\\
\textbf{Claim 1}: For any $\epsilon>0$ and $t>0$, 
\begin{equation}\label{lower}
\inf_{X}\tilde R_\epsilon\ge -\frac{3n}{\min\{t,1\}}
\end{equation}
\begin{proof}[Proof of Claim 1]
First we have the evolution of $\tilde R_\epsilon$ along \eqref{TCKRF.1} (see e.g. \cite[Proposition 4.1]{Ed15}):
\begin{equation}\label{lower1}
(\partial_t-\Delta_{\omega_\epsilon(t)})\tilde R_\epsilon=|Ric(\omega_\epsilon(t))-\frac{1}{T_{max}}\omega_0-A_\epsilon|_{\omega_{\epsilon}(t)}^2+\tilde R_\epsilon,
\end{equation}
which, combining with $|Ric(\omega_\epsilon(t))-\frac{1}{T_{max}}\omega_0-A_\epsilon|_{\omega_{\epsilon}(t)}^2\ge\frac{1}{n}\tilde R_\epsilon^2$, implies
\begin{equation}\label{lower2}
(\partial_t-\Delta_{\omega_\epsilon(t)})(e^t(\tilde R_\epsilon+n))\ge \frac{1}{n}e^t(\tilde R_\epsilon+n)^2.
\end{equation}
To avoid involving the uniform bound for $\tilde R_\epsilon(0)$, we use a trick similar to e.g. \cite[Lemma 3.2]{ST17} and further consider
\begin{equation}\label{lower3}
(\partial_t-\Delta_{\omega_\epsilon(t)})(te^t(\tilde R_\epsilon+n))\ge \frac{1}{n}te^t(\tilde R_\epsilon+n)^2+e^t(\tilde R_\epsilon+n)
\end{equation}
We now apply the maximum principle in \eqref{lower3}. Assume $te^t(\tilde R_\epsilon+n)$ achieves the minimal value on $X\times[0,1]$ at $(x',t')\in X\times[0,1]$. If $t'=0$, $(te^t(\tilde R_\epsilon+n))(x',t')=0$ is uniformly bounded from below; otherwise $t'>0$, then by \eqref{lower3} we know, at $(x',t')$,
\begin{equation}\label{lower4.1}
t(\tilde R_\epsilon+n)^2\le -n(\tilde R_\epsilon+n),
\end{equation}
which in particular implies $(\tilde R_\epsilon+n)(x',t')\le0$. We may assume $$(\tilde R_\epsilon+n)(x',t')<0$$ (otherwise we are done), and so by \eqref{lower4.1} we find
\begin{equation}\label{lower5}
t(\tilde R_\epsilon+n)\ge -n,
\end{equation}
and so $(te^t(\tilde R_\epsilon+n))(x',t')\ge-en$. In conclusion, we have 
\begin{equation}\label{lower5.1}
\inf_{X\times[0,1]}(te^t(\tilde R_\epsilon+n))\ge -en,
\end{equation}
which implies
\begin{equation}\label{lower5.2}
\inf_{X\times[0,1]}(t(\tilde R_\epsilon+n))\ge -en,
\end{equation}
and so 
$\inf_X(\tilde R_\epsilon+n))(1)\ge-n$. Now, back to \eqref{lower2}, we easily find that for any $t\ge1$, 
$$\inf_X(e^t(\tilde R_\epsilon+n))(t)\ge\inf_X(e^t(\tilde R_\epsilon+n))(1)\ge-en,$$
and 
\begin{equation}\label{lower5.3}
\inf_{X\times[1,\infty)}(\tilde R_\epsilon+n)\ge-n.
\end{equation}
Combining \eqref{lower5.2} and \eqref{lower5.3}, Claim 1 is proved.
\end{proof}

Next we try to get an upper bound for $\tilde R_\epsilon$. 
Recall that \cite[Propositions 3.1,3.2,6.1]{Ed15} provides a constant $C\ge1$ such that for any small $\epsilon>0$, there holds
\begin{equation}
\sup_{X\times[0,\infty)}(|\varphi_\epsilon|+|\partial_t\varphi_\epsilon|+tr_{\omega_\epsilon(t)}f^*\chi)\le C.
\end{equation}
It suffices to remark that the twisted term $\frac{1}{T_{max}}\omega_0$ is smooth and semipositive and hence does not cause any trouble.

We also recall some useful inequalities. Since $\chi$ is a K\"ahler metric on $Y$, which in particular has bounded bisectional curvature, by arguments similar to Lemma \ref{C2.1} (or see \cite[Section 6]{Ed15}) we have a constant $C\ge1$ such that for all positive $\epsilon$,
\begin{align}\label{P2.2}
&(\partial_t-\Delta_{\omega_\epsilon(t)})tr_{\omega_\epsilon(t)}f^*\chi\le C-C^{-1}|\nabla tr_{\omega_\epsilon(t)} f^*\chi|_{\omega_\epsilon(t)}^2,
\end{align}
where we have used $A_\epsilon\ge-Cf^*\chi$ and $\frac{1}{T_{max}}\omega_0\ge0$.

Set $u_\epsilon:=\partial_t\varphi_\epsilon+\varphi_\epsilon+\delta\eta_\epsilon$. Then by a direct computation we have, on $X\times[0,\infty)$,
\begin{equation}
(\partial_t-\Delta)u_\epsilon=tr_{\omega_{\epsilon}(t)}f^*\chi-k\le C.
\end{equation}
To simplify the notation we set $|\cdot|:=|\cdot|_{\omega_\epsilon(t)}$.

We first bound $|\nabla u_\epsilon|$. A direct computation gives
\begin{align}
&(\partial_t-\Delta)|\nabla u_\epsilon|^2\nonumber\\
&=|\nabla u_\epsilon|^2+2Re\langle\nabla tr_{\omega_\epsilon(t)} f^*\chi,\overline\nabla u_\epsilon\rangle_{\omega_\epsilon(t)}-|\nabla\nabla u_\epsilon|^2-|\nabla\overline\nabla u_\epsilon|^2-(\frac{1}{T_{max}}\omega_0+A_\epsilon)(\nabla u_\epsilon,\overline\nabla u_\epsilon),
\end{align}
which, combining with $A_\epsilon\ge-Cf^*\chi$ and $\frac{1}{T_{max}}\omega_0\ge0$, in particular implies that
\begin{align}\label{P2.2.3}
(\partial_t-\Delta)|\nabla u_\epsilon|^2\le-|\nabla\nabla u_\epsilon|^2-|\nabla\overline\nabla u_\epsilon|^2+C|\nabla u_\epsilon|^2+|\nabla tr_{\omega_\epsilon(t)} f^*\chi|^2.
\end{align}

Now we fix a sufficiently large constant $B>4$ such that
\begin{equation}
2|u_\epsilon|\le B,
\end{equation}
and so,
\begin{equation}\label{footnote1}
2<\frac{B}{2}\le B-u_\epsilon\le2B.
\end{equation}
Fix a constant $s\in(0,1)$ and set\footnote{We point out that in previous works on bounding scalar curvature along K\"ahler-Ricci flow (see \cite{SeT,ST16,Zo,Ed15}), people usually choose $s=1$ to define this quantity $E$. Here we choose $s<1$ to slightly simplify some arguments.} $E_\epsilon:=\frac{|\nabla u_\epsilon|^2}{(B-u_\epsilon)^s}$. Compute
\begin{align}\label{P2.4}
&(\partial_t-\Delta_{\omega_\epsilon(t)})E_\epsilon\nonumber\\
&=\frac{(\partial_t-\Delta_{\omega_\epsilon(t)})|\nabla u_\epsilon|^2}{(B-u_\epsilon)^s}+\frac{s|\nabla u_\epsilon|^2(\partial_t-\Delta_{\omega_\epsilon(t)}) u_\epsilon}{(B-u_\epsilon)^{s+1}}-\frac{2sRe\langle\nabla |\nabla u_\epsilon|^2,\nabla u_\epsilon\rangle_{\omega_\epsilon(t)}}{(B-u_\epsilon)^{s+1}}-\frac{s(s+1)|\nabla u_\epsilon|^4}{(B-u_\epsilon)^{s+2}}\nonumber\\
&\le\frac{-|\nabla\nabla u_\epsilon|^2-|\nabla\overline\nabla u_\epsilon|^2+2|\nabla u_\epsilon|^2+|\nabla tr_{\omega_\epsilon(t)} f^*\chi|^2}{(B-u_\epsilon)^s}+\frac{sC|\nabla u_\epsilon|^2}{(B-u_\epsilon)^{s+1}}\nonumber\\
&+\frac{2s|\langle\nabla |\nabla u_\epsilon|^2,\nabla u_\epsilon\rangle_{\omega_\epsilon(t)}|}{(B-u_\epsilon)^{s+1}}-\frac{s(s+1)|\nabla u_\epsilon|^4}{(B-u_\epsilon)^{s+2}}.
\end{align}
Note that
\begin{align}\label{P2.5}
\frac{2s|\langle\nabla |\nabla u_\epsilon|^2,\nabla u_\epsilon\rangle_{\omega_\epsilon(t)|}}{(B-u_\epsilon)^{s+1}}&\le\frac{2s|\nabla u_\epsilon|^2(|\nabla\nabla u_\epsilon|+|\nabla\bar\nabla uz-\epsilon|)}{(B-u_\epsilon)^{s+1}}\nonumber\\
&\le s\left(\frac{2(|\nabla\nabla u_\epsilon|^2+|\nabla\overline\nabla u_\epsilon|^2)}{2s(B-u_\epsilon)^s}+\frac{2s|\nabla u_\epsilon|^4}{(B-u_\epsilon)^{s+2}}\right)\nonumber\\
&=\frac{|\nabla\nabla u_\epsilon|^2+|\nabla\overline\nabla u_\epsilon|^2}{(B-u_\epsilon)^s}+\frac{2s^2|\nabla u_\epsilon|^4}{(B-u_\epsilon)^{s+2}}.
\end{align}
Putting \eqref{P2.5} into \eqref{P2.4} gives\footnote{Note that choosing the positive number $s$ to be in $(0,1)$ helps to make sure the factor $s-s^2$ appeared in \eqref{P2.6} is positive and avoid an additional trick used in \cite{ST16,Ed15}.}
\begin{align}\label{P2.6}
(\partial_t-\Delta_{\omega_\epsilon(t)})E_\epsilon&\le C|\nabla u_\epsilon|^2+|\nabla tr_{\omega_\epsilon(t)} f^*\chi|^2-(B-u_\epsilon)^{-(s+2)}(s-s^2)|\nabla u_\epsilon|^4\nonumber\\
&\le C|\nabla u_\epsilon|^2+|\nabla tr_{\omega_\epsilon(t)} f^*\chi|^2-(2B)^{-(s+2)}(s-s^2)|\nabla u_\epsilon|^4
\end{align}
Now consider $F_\epsilon:=E_\epsilon+C_1tr_{\omega_\epsilon(t)} f^*\chi$. Combining \eqref{P2.2} and \eqref{P2.6} gives
\begin{equation}\label{P2.7}
(\partial_t-\Delta_{\omega_\epsilon(t)})F_\epsilon\le C+C|\nabla u_\epsilon|^2-(2B)^{-(s+2)}(s-s^2)|\nabla u_\epsilon|^4.
\end{equation}
To get an upper bound for $F_\epsilon$, one may like to apply the maximum principle in $\eqref{P2.7}$ (see last step in \cite[Section 7]{Ed15} or \cite[Section 6]{No}). Note that, arguing in such way, since the maximal value may be achieved at $t=0$, one may need a uniform bound for $|\nabla u_{\epsilon}|(0)=|\nabla\log\frac{(\omega_{0,\epsilon}^*)^n}{(|S|^2_h+\epsilon^2)^{-(1-\beta)}\Omega}|$. 
To avoid involving such bound, here we use a trick similar to e.g. \cite[Lemma 3.2]{ST17} and further consider
\begin{align}\label{P2.8}
(\partial_t-\Delta_{\omega_\epsilon(t)})(tF_\epsilon)&\le F_{\epsilon}+t(C+C|\nabla u_\epsilon|^2-(2B)^{-(s+2)}(s-s^2)|\nabla u_\epsilon|^4)\nonumber\\
&\le C(t+1)+C(t+1)|\nabla u_\epsilon|^2-t(2B)^{-(s+2)}(s-s^2)|\nabla u_\epsilon|^4.
\end{align}
Assume the maximal value of $tF_\epsilon$ on $X\times[0,1]$ is achieved at $(x',t')\in X\times[0,1]$. If $t'=0$, $(tF_\epsilon)(x',t')=0$ is bounded from above; otherwise $t'>0$, by \eqref{P2.8} we find a constant $C\ge1$ independent on $\epsilon$ such that, at $(x',t')$, 
$$t|\nabla u_\epsilon|^4\le C(t+1)|\nabla u_\epsilon|^2+C(t+1).$$
Note that $t'\in(0,1]$. We have, at $(x',t')$,
\begin{align}
(t|\nabla u_\epsilon|^2)^2&\le Ct(t+1)+Ct(t+1)|\nabla u_\epsilon|^2\nonumber\\
&\le C+C(t|\nabla u_\epsilon|^2)
\end{align}
and so
$$(t|\nabla u_\epsilon|^2)(x',t')\le C.$$
Plugging into $tF_\epsilon$ we find
$$(tF_\epsilon)(x',t')=\left(\frac{t|\nabla u_\epsilon|^2}{(B-u_\epsilon)^s}+tC_1tr_{\omega_\epsilon(t)}f^*\chi\right)(x',t')\le C.$$
Therefore, we have find a constant $C\ge1$ independent on $\epsilon$ such that
\begin{equation}\label{upper11}
\sup_{X\times[0,1]}(tF_\epsilon)\le C,
\end{equation}
which implies
\begin{equation}\label{upper11.1}
\sup_{X\times[0,1]}(t|\nabla u_\epsilon|^2)\le C.
\end{equation}

In particular, we have a constant $C\ge1$ independent on $\epsilon$ such that
\begin{equation}\label{upper1}
\sup_{X}F_\epsilon(1)\le C.
\end{equation}
Having \eqref{upper1}, we can now apply the maximum principle in \eqref{P2.7} to conclude that
\begin{equation}
\sup_{X\times[1,\infty)}F_{\epsilon}\le C\nonumber
\end{equation}
and so
\begin{equation}\label{upper2}
\sup_{X\times[1,\infty)}|\nabla u_{\epsilon}|^2\le C
\end{equation}
for some constant $C\ge1$ independent on $\epsilon$.

Combining \eqref{P2.2}, \eqref{P2.2.3} and \eqref{upper2}, we have constants $C_1,C\ge1$ independent on $\epsilon$ such that
\begin{equation}\label{P2.4.0}
(\partial_t-\Delta_{\omega_\epsilon(t)})(t(|\nabla u_\epsilon|^2+C_1 tr_{\omega_\epsilon(t)} f^*\chi))\le-t|\nabla\nabla u_\epsilon|^2-t|\nabla\overline\nabla u_\epsilon|^2+C.
\end{equation}
on $(X\setminus D)\times[0,1]$, and
\begin{equation}\label{P2.4.1}
(\partial_t-\Delta_{\omega_\epsilon(t)})(|\nabla u_\epsilon|^2+C_1 tr_{\omega_\epsilon(t)} f^*\chi)\le-|\nabla\nabla u_\epsilon|^2-|\nabla\overline\nabla u_\epsilon|^2+C.
\end{equation}
on $(X\setminus D)\times[1,\infty)$.

Next we try to bound $-\Delta_{\omega_\epsilon(t)}$ from above. Recall that
\begin{equation}\label{P2.8.1}
(\partial_t-\Delta_{\omega_\epsilon(t)})\Delta_{\omega_\epsilon(t)} u_\epsilon=\Delta_{\omega_\epsilon(t)} u_\epsilon+\langle Ric(\omega_\epsilon(t))-\frac{1}{T_{max}}\omega_0-A_\epsilon,\sqrt{-1}\partial\bar\partial u_\epsilon\rangle_{\omega_\epsilon}+\Delta_{\omega_\epsilon(t)} tr_{\omega_\epsilon(t)}f^*\chi.
\end{equation}
Using
\begin{equation}\label{eq1}
\sqrt{-1}\partial \bar\partial u_\epsilon=-Ric(\omega_\epsilon)-f^*\chi+\frac{1}{T_{max}}\omega_0+A_\epsilon,
\end{equation}
we see that the second term in \eqref{P2.8.1}
\begin{align}\label{P2.9}
\langle Ric(\omega_\epsilon(t))-\frac{1}{T_{max}}\omega_0-A_\epsilon,\sqrt{-1}\partial\bar\partial u_\epsilon\rangle_{\omega_\epsilon(t)}&=\langle-\sqrt{-1}\partial\bar\partial u_\epsilon-f^*\chi,\sqrt{-1}\partial\bar\partial u_\epsilon\rangle_{\omega_\epsilon(t)}\nonumber\\
&=-|\nabla\bar\nabla u_\epsilon|^2-\langle f^*\chi,\sqrt{-1}\partial\bar\partial u_\epsilon\rangle_{\omega_\epsilon(t)}\nonumber\\
&\ge-\frac{3}{2}|\nabla\bar\nabla u_\epsilon|^2-\frac{1}{2}|f^*\chi|^2\nonumber\\
&\ge-\frac{3}{2}|\nabla\bar\nabla u_\epsilon|^2-\frac{n}{2}(tr_{\omega_\epsilon(t)} f^*\chi)^2\nonumber\\
&\ge-\frac{3}{2}|\nabla\bar\nabla u_\epsilon|^2-C,
\end{align}
and the third term in \eqref{P2.8.1}
\begin{align}
\Delta_{\omega_\epsilon(t)} tr_{\omega_\epsilon(t)} f^*\chi&=-(\partial_t-\Delta_{\omega_\epsilon(t)})tr_{\omega_\epsilon(t)} f^*\chi+\partial_t tr_{\omega_\epsilon(t)} f^*\chi\nonumber\\
&\ge -C+C^{-1}_1|\nabla tr_{\omega_\epsilon(t)} f^*\chi|^2+\langle Ric(\omega_\epsilon(t))-\frac{1}{T_{max}}\omega_0-A_{\epsilon},f^*\chi\rangle_{\omega_\epsilon(t)}+tr_{\omega_\epsilon(t)} f^*\chi\nonumber\\
&\ge-C+\langle Ric(\omega_\epsilon(t))-\frac{1}{T_{max}}\omega_0-A_{\epsilon},f^*\chi\rangle_{\omega_\epsilon(t)}\nonumber\\
&=-C+\langle-\sqrt{-1}\partial\bar\partial u_\epsilon-f^*\chi,f^*\chi\rangle_{\omega_\epsilon(t)}\nonumber\\
&\ge-C-\frac{1}{2}|\nabla \overline\nabla u_\epsilon|^2-\frac{3}{2}|f^*\chi|^2\nonumber\\
&\ge-C-\frac{1}{2}|\nabla \overline\nabla u_\epsilon|^2.
\end{align}
Then we arrive at
\begin{align}
(\partial_t-\Delta_{\omega_\epsilon(t)})\Delta_{\omega_\epsilon(t)} u_\epsilon\ge\Delta_{\omega_\epsilon(t)} u_\epsilon-2|\nabla \overline\nabla u_\epsilon|^2-C.
\end{align}
Using \eqref{P2.4.0} and \eqref{P2.4.1}, we see
\begin{align}\label{P2.10.0}
&(\partial_t-\Delta_{\omega_\epsilon(t)})\left(t(\Delta_{\omega_\epsilon(t)} u_\epsilon-3(|\nabla u_\epsilon|^2+C_1 tr_{\omega_\epsilon(t)} f^*\chi))\right)\nonumber\\
&\ge (t+1)\Delta_{\omega_\epsilon(t)} u_\epsilon+t|\nabla\overline\nabla u_\epsilon|^2-C\nonumber\\
&\ge (t+1)\Delta_{\omega_\epsilon(t)} u_\epsilon+\frac{1}{n}t(\Delta_{\omega_\epsilon(t)} u_\epsilon)^2-C.
\end{align}
on $X\times[0,1]$, and 
\begin{align}\label{P2.10}
&(\partial_t-\Delta_{\omega_\epsilon(t)})\left(\Delta_{\omega_\epsilon(t)} u_\epsilon-3(|\nabla u_\epsilon|^2+C_1 tr_{\omega_\epsilon(t)} f^*\chi)\right)\nonumber\\
&\ge\Delta_{\omega_\epsilon(t)} u_\epsilon+|\nabla\overline\nabla u_\epsilon|^2-C\nonumber\\
&\ge\Delta_{\omega_\epsilon(t)} u_\epsilon+\frac{1}{n}(\Delta_{\omega_\epsilon(t)} u_\epsilon)^2-C
\end{align}
on $X\times[1,\infty)$.
To bound $\Delta_{\omega_\epsilon(t)}u_\epsilon$ from below (see \cite[Section 8]{Ed15} or \cite[Section 7]{No} for related discussions), we also try to avoid using a uniform bound for $(\Delta_{\omega_\epsilon(t)}u_\epsilon)(0)=\Delta_{\omega_\epsilon(t)}\log\frac{(\omega_{0,\epsilon}^*)^n}{(|S|_h^2+\epsilon^2)^{-(1-\beta)}\Omega}$. For convenience, we set $H_\epsilon:=t(\Delta_{\omega_\epsilon(t)} u_\epsilon-3(|\nabla u_\epsilon|^2+C_1 tr_{\omega_\epsilon(t)} f^*\chi))$, which by \eqref{P2.10.0} satisfies
\begin{align}\label{P2.10.1}
&(\partial_t-\Delta_{\omega_\epsilon(t)})\left(t(\Delta_{\omega_\epsilon(t)} u_\epsilon-3(|\nabla u_\epsilon|^2+C_1 tr_{\omega_\epsilon(t)} f^*\chi))\right)\nonumber\\
&\ge (t+1)\Delta_{\omega_\epsilon(t)} u_\epsilon+\frac{1}{n}t(\Delta_{\omega_\epsilon(t)} u_\epsilon)^2-C
\end{align}
$X\times[0,1]$.
We assume the minimal value of $H_\epsilon$ on $X\times[0,1]$ is achieved at $(x',t')\in X\times[0,1]$. If $t'=0$, $H_\epsilon(x',t')=0$ is bounded from below; otherwise $t'>0$, then by maximum principle in \eqref{P2.10.1} we find, at $(x',t')$,
$$t(\Delta_{\omega_\epsilon(t)}u_\epsilon)^2\le-(t+1)\Delta_{\omega_\epsilon(t)}u_\epsilon+C,$$
and so 
$$(t\Delta_{\omega_\epsilon(t)}u_\epsilon)^2\le-(t+1)(t\Delta_{\omega_\epsilon(t)}u_\epsilon)+Ct.$$
We may assume $(\Delta_{\omega_\epsilon(t)}u_\epsilon)(x',t')<0$ (otherwise we are done); then by noting that $t\in(0,1]$ we have
$$(t\Delta_{\omega_\epsilon(t)}u_\epsilon)^2\le-2(t\Delta_{\omega_\epsilon(t)}u_\epsilon)+C,$$
from which we conclude that
$$(t\Delta_{\omega_\epsilon(t)}u_\epsilon)(x',t')\ge -C.$$
Plugging into $H_\epsilon$ gives
$$H_\epsilon(x',t')=\left(t(\Delta_{\omega_\epsilon(t)} u_\epsilon-3(|\nabla u_\epsilon|^2+C_1 tr_{\omega_\epsilon(t)} f^*\chi))\right)(x',t')\ge-C.$$
Therefore, we can choose a constant $C\ge1$ independent on $\epsilon$ such that
$$\inf_{X\times[0,1]}H_\epsilon\ge-C,$$
which, combining with \eqref{upper11.1}, implies
\begin{equation}\label{lower3.1}
\inf_{X\times[0,1]}(t\Delta_{\omega_\epsilon(t)}u_\epsilon)\ge-C.
\end{equation}
In particular,
\begin{equation}\label{lower4}
\inf_X(\Delta_{\omega_\epsilon(t)} u_\epsilon-3(|\nabla u_\epsilon|^2+C_1 tr_{\omega_\epsilon(t)} f^*\chi))(1)\ge -C.
\end{equation}

Having \eqref{lower4}, we can easily apply the maximum principle in \eqref{P2.10} to find a constant $C\ge1$ independent on $\epsilon$ such that
\begin{equation}\label{lower5}
\inf_{X\times[1,\infty)}\Delta_{\omega_\epsilon(t)}u_\epsilon\ge-C.
\end{equation}
Combining \eqref{lower3.1} and \eqref{lower5} gives, for any $t>0$,
\begin{equation}\label{lower5.1}
\inf_{X}\Delta_{\omega_\epsilon(t)}u_\epsilon\ge-\frac{C}{\min\{t,1\}}.
\end{equation}
Finally, recall from \eqref{eq1} that
\begin{equation}
\tilde R_\epsilon=-tr_{\omega_\epsilon(t)}f^*\chi-\Delta_{\omega_\epsilon(t)}u_\epsilon.
\end{equation}
Therefore, by \eqref{lower5.1} there exists a constant $C\ge1$ independent on $\epsilon$ such that for any $t>0$
\begin{equation}\label{lower6}
\sup_{X}\tilde R_\epsilon\le \frac{C}{\min\{t,1\}}.
\end{equation}

Combining \eqref{lower} and \eqref{lower6} gives a constant $C\ge1$ independent on $\epsilon$ such that for any $t>0$,
\begin{equation}\label{scal1}
\sup_{X}|\tilde R_\epsilon|\le \frac{C}{\min\{t,1\}}.
\end{equation}
By letting $\epsilon\to0$, Lemma \ref{scal} is proved.

\end{proof}
\subsection{Proof of Theorem \ref{scal_thm}}\label{scal.1} In this subsection, we explain how to apply the above arguments for Lemma \ref{scal} to prove Theorem \ref{scal_thm}.
\begin{proof}[Proof of Theorem \ref{scal_thm}]
It suffices to prove the conclusion in Theorem \ref{scal_thm} for the following normalized version of \eqref{CKRF}:
\begin{equation}\label{NCKRF.1}
\left\{
\begin{aligned}
\partial_t\omega(t)&=-Ric(\omega(t))-\omega(t)+2\pi(1-\beta)[D]\\
\omega(0)&=\omega_0^*,
\end{aligned}
\right.
\end{equation}
here, without loss of any generality we assume there is only one irreducible component $D$ in the cone divisor. We have the smooth approximation \eqref{TCKRF.1} and \eqref{TCCMA.1} as before. For the fixed $t_0\in(0,T_{max})$, we fix a K\"ahler metric $\hat\omega_{t_0}\in e^{-t_0}[\omega_0]+(1-e^{-t_0})2\pi(c_1(K_X)+(1-\beta)[D])$. Also note that $\varphi_\epsilon$ and $\partial_t\varphi_\epsilon$ are uniformly bounded on $X\times[0,t_0]$ (see \cite{Sh,Wy}). In the above arguments for Lemma \ref{scal}, we replace $Y$ by $X$, $f$ by the identity map $X\to X$ and $\chi$ by $\hat\omega_{t_0}$, and then those arguments apply and Theorem \ref{scal_thm} follows.
\end{proof}

\section{Local $C^0$-convergence away from cone divisor}\label{sect_C0}
In this section, we will prove $C^0$-convergence away from cone divisor $D$. The strategy used here is taken from Tosatti-Weinkove-Yang \cite{TWY}.
\par In this section, similar to discussions in Sections \ref{weak_conv} and \ref{equiv}, we will directly work with the conical equations without passing to a smooth approximation. Let's begin with the following

\begin{lem}\label{prop_conv.1}
There exists a constant $C\ge1$ such that for any $t\in[1,\infty)$,
\begin{equation}\label{conv}
\sup_{X\setminus D}|\partial_t\varphi+\varphi(t)+\delta|S|^{2\beta}_h-f^*\psi|\le Ce^{-\frac{1}{4}t}.
\end{equation}
\end{lem}

\begin{proof}
Recall Proposition \ref{prop_conv}: there exists a contant $\tilde C$ such that for any $t\in[0,\infty)$
\begin{equation}\label{local0}
\sup_{X\setminus D}|\varphi+\delta|S|^{2\beta}_h-f^*\psi|\le \tilde Ce^{-\frac{1}{4}t}.
\end{equation}
Also recall
\begin{equation}
\partial_t(\partial_t\varphi)=\partial_t\varphi-\tilde R-k.
\end{equation}\label{local2}
So, by Lemma \ref{scal} we can fix a constant $C_0\ge1$ such that $\sup_{(X\setminus D)\times[1,\infty)}|\partial_t(\partial_t\varphi)|\le C_0$. We set $C_1:=\sqrt{8\tilde CC_0+1}$. It suffices to show
\begin{equation}\label{local2.1}
\sup_{X\setminus D}|\partial_t\varphi|\le C_1e^{-\frac{1}{4}t}.
\end{equation}
To this end, we claim that\\
\textbf{Claim} For any $\lambda\in(0,\lambda_0]$ and $t\in[1,\infty)$,
\begin{equation}\label{local3}
\sup_{X\setminus D}(\partial_t\varphi+\lambda\log|S|^2_h)\le C_1e^{-\frac{1}{4}t}.
\end{equation}
\begin{proof}[Proof of Claim]
We make use of an argument in \cite{TWY} to conclude. Suppose there exists a sequence $(x_k,t_k)\in (X\setminus D)\times[1,\infty)$ with $t_k\to\infty$ as $k\to\infty$ such that
$$(\partial_t\varphi+\lambda\log|S|^2_h)(x_k,t_k)\ge C_1e^{-\frac{1}{4}t_k},$$
which in particular implies that $x_k$ in fact will be contained in some fixed $K\subset\subset X\setminus D$, since $\partial_t\varphi$ is uniformly bounded.
Set $\theta_k:=\frac{k}{2C_0}e^{-\frac{1}{4}t_k}$, where $C_0$ is a constant such that $\sup_{(X\setminus D)\times[1,\infty)}|\partial_t(\partial_t\varphi)|\le C_0$ defined as before. Then, $(\partial_t\varphi+\lambda\log|S|^2_h)(x_k,t)\ge\frac{C_1}{2}e^{-\frac{1}{4}t_k}$ for all $t\in[t_k,t_k+\theta_k]$. On the one hand we have
\begin{align}\label{local4}
&(\varphi+\delta|S|^{2\beta}_h-f^*\psi)(x_k,t_k+\theta_k)-(\varphi+\delta|S|^{2\beta}_h-f^*\psi)(x_k,t_k)\nonumber\\
&\le\sup_{X\setminus D}|\varphi+\delta|S|^{2\beta}_h-f^*\psi|(t_k+\theta_k)+\sup_{X\setminus D}|\varphi+\delta|S|^{2\beta}_h-f^*\psi|(t_k)\nonumber\\
& \le \tilde C(e^{-\frac{3}{4}(t_k+\theta_k)}+e^{-\frac{3}{4}t_k});\nonumber\\
&\le 2\tilde Ce^{-\frac{3}{4}t_k}.
\end{align}
On the other hand,
\begin{align}\label{local5}
&(\varphi+\delta|S|^{2\beta}_h-f^*\psi)(x_k,t_k+\theta_k)-(\varphi+\delta|S|^{2\beta}_h-f^*\psi)(x_k,t_k)\nonumber\\
&\ge\int_{t_k}^{t_k+\theta_k}(\partial_t\varphi)(x_k,t)dt\nonumber\\
&\ge\int_{t_k}^{t_k+\theta_k}(\partial_t\varphi+\lambda\log|S|^2_h)(x_k,t)dt\nonumber\\
&\ge\theta_k\cdot \frac{C_1}{2}e^{-\frac{1}{4}t_k}\nonumber\\
&=\frac{C_1^2}{4C_0}e^{-\frac{1}{2}t_k}\nonumber\\
&\ge\frac{C_1^2}{4C_0}e^{-\frac{3}{4}t_k}.
\end{align}
By combining with \eqref{local4} and \eqref{local5} we get a contradiction:
$$8\tilde CC_0+1=C^2_1\le 8\tilde CC_0.$$
Therefore, Claim is proved.
\end{proof}
Let $\lambda\to0$ in \eqref{local3}, we get, for all $t\in[1,\infty)$,
\begin{equation}\label{local6}
\sup_{X\setminus D}\partial_t\varphi\le C_1e^{-\frac{1}{4}t}.
\end{equation}
Similarly, we can get, for all $t\in[1,\infty)$,
\begin{equation}\label{local7}
\inf_{X\setminus D}\partial_t\varphi\ge -C_1e^{-\frac{1}{4}t}.
\end{equation}
Having \eqref{local6} and \eqref{local7}, \eqref{local2.1} follows and Lemma \ref{prop_conv.1} is proved.
\end{proof}

\begin{lem}\label{local5.1}
There exist two positive constants $C,\gamma$ such that for any $t\in[1,\infty)$,
\begin{equation}\label{local6}
\sup_{X\setminus D}(|S|^{2\gamma}_h(tr_{\omega(t)}f^*\overline\chi-k))\le Ce^{-\frac{1}{8}t}.
\end{equation}
\end{lem}

\begin{proof}
Recall we have 
$$(\partial_t-\Delta_{\omega(t)})(\partial_t\varphi+\varphi+\delta|S|^{2\beta}_h-f^*\psi)=tr_{\omega(t)}f^*\overline\chi-k$$
and, for some $\gamma>0$,
$$(\partial_t-\Delta_{\omega(t)})tr_{\omega(t)}\overline\chi\le tr_{\omega(t)}\overline\chi\le+|S|^{2\gamma}_h(tr_{\omega(t)}\overline\chi)^2\le C|S|^{2\gamma}_h,$$
where $|S|^{2\gamma}_h$ is an uniform upper bound for bisectional curvature of $\overline\chi$ on $Y\setminus D'$ and we have used $tr_{\omega(t)}f^*\overline\chi$ is uniformly bounded by Lemma \ref{C2.1} (note that $\overline\chi$ and $\chi^*$ is uniformly equivalent). Then one can apply arguments in \cite[Lemma 3.4]{TWY} to conclude \eqref{local6}.
\par Lemma \ref{local5.1} is proved.
\end{proof}

\begin{lem}
For any given $K'\subset\subset Y\setminus D'$, and $l\in \mathbb Z_{\ge1}$, there exists a constant $C=C_{K,l}\ge1$ such that for any $y\in K'$ and $t\in[1,\infty)$,
\begin{equation}\label{local8}
|e^t\omega(t)|_{X_y}|_{C^l(X_y,\omega_0|_{X_y})}\le C_{K,l}.
\end{equation}
\end{lem}
\begin{proof}
This lemma can be checked by the same arguments in \cite[Theorem 1.1]{ToZyg}, as we have Proposition \ref{prop2}. To see that arguments in \cite[page 2933-2934]{ToZyg} work in our twisted case, it suffices to observe that, after pulling back the twisted term $\frac{1}{T_{max}}\omega_0$ by the "stretching map" $F_k$ defined in \cite[page 2933]{ToZyg}, we get a smooth $(1,1)$-form $\frac{1}{T_{max}}F^*_k\omega_0$, whose $C^1$-norm with respect to a local Euclidean metric is uniformly bounded. Then we can get the desired conclusions.
\end{proof}

With all the above preparations, we now conclude the main result in this section.
\begin{prop}\label{prop3}
There exists a positive constant $\varepsilon_0$ such that, for any $K\subset\subset X\setminus D$ there exists a constant $C_K\ge1$ such that, for any $t\in[1,\infty)$,
\begin{equation}
|\omega(t)-f^*\overline\chi|_{C^0(K,\omega_0)}\le C_Ke^{-\varepsilon_0t}.
\end{equation}
\end{prop}

\begin{proof}
Having the above results, one can first prove $$|e^t\omega(t)|_{X_y}-\overline\omega_0|_{X_y}|_{C^0(X_y,\omega_0|_{X_y})}\le C_Ke^{-\varepsilon t}$$ for any $y\in K$, and then $$|\omega(t)-(e^{-t}\overline\omega_0+f^*\overline\chi)|_{C^0(K,\omega_0)}\le C_Ke^{-\varepsilon_0t}$$ by the same arguments in \cite[Section 2.5]{TWY}. We omit details here. Therefore, Proposition \ref{prop3} follows.
\end{proof}

\section{Gromov-Hausdorff convergence}\label{sect_gh}
We are going to prove item (2.3) in Theorem \ref{main_thm}, i.e. Gromov-Hausdorff convergence.
\par Throughout this section, we will always assume $dim(Y)=1$ and so $D'$ is a single point $o\in Y$. To obtain Gromov-Hausdorff convergence, the key point is to apply Proposition \ref{prop2} to bound diameter of the cone divisor $X_o$. 
We set $B'_\delta:=\{y\in Y|d_{\chi}(y,o)<\delta\}$ and $B_\delta:=f^{-1}(B_\delta)$. Assume without loss of generality that $\delta$ is always small enough such that $B_\delta$ is the standard disc in $\mathbb C$ and $o=0\in\mathbb C$. Fix an integer $L\ge1$ such that 
\begin{equation}
diam(B_{\varepsilon^L},\bar d)\le \frac{\varepsilon}{4}.
\end{equation}
The key observation is the following
\begin{lem}\label{gh0}
There exists a constant $C\ge1$ such that, for any $\varepsilon>0$, there exists a positive constant $T=T_\varepsilon$ such that for all $t\ge T$,
\begin{equation}
diam(B_{\varepsilon^L},d_t)\le C\varepsilon
\end{equation}
\end{lem}

\begin{proof}
The proof will make use of Proposition \ref{prop2} crucially. Firstly, since $X_o$ is compact, we choose a family of local charts $\{U_i,(y,z^1,...,z^{n-1})\}_{i=1}^{I}$, such that the chart $U_i$ is centered at some point $x_i\in X_o$, $B_{\varepsilon^L}\subset \cup_{i=1}^IU_i$ and $f(y,z^1,...,z^{n-1})=y$. For any two points $x',x''\in B_{\varepsilon^L}\setminus X_o$, we may assume $x'=(y',z')\in U_1$ and $x''=(y'',z'')\in U_2$. Then when we restrict to the subset $Y_{z'}:=\{(y,z')\in U_1|y\in B'_{\varepsilon^L}\}$, we know from Proposition \ref{prop2} that $\omega(t)|_{Y_{z'}}$ is a conical metric on $Y_{z'}$ with cone angle $2\pi\beta$ at $(0,z')$. Therefore, we can easily choose a point $\bar y'\in \partial B'_{\varepsilon^L}$, and set $\bar x':=(\bar y',z')\in Y_{z'}$, such that 
\begin{equation}\label{gh1}
d_t(x',\bar x')\le \hat C\varepsilon,
\end{equation}
here $\hat C$ is a positive constant only depends on the constant $C$ in Proposition \ref{prop2}, \eqref{C2.3} (in particular, $\hat C$ does not depend on $x',x''$ and $t$). Similarly, one can find a point $\bar x'':=(\bar y'',z'')$ with $\bar y''\in\partial B'_{\varepsilon^L}$ such that
\begin{equation}\label{gh2}
d_t(x'',\bar x'')\le \hat C\varepsilon,
\end{equation} 
On the other hand, since $\overline\chi$ is a conical K\"ahler metric, the diameters of fibers $X_y$ away from $X_o$ uniformly collapse in the rate $e^{-\frac{1}{2}t}$ by Proposition \ref{prop2} and $\omega(t)\to f^*\overline\chi$ in $C^0(f^{-1}(Y\setminus B'_{\frac{1}{2}\varepsilon^L},\omega_0)$-topology by Proposition \ref{prop3}, one can use arguments in \cite[Section 3]{ZyZz1} to connect $\bar x',\bar x''$ by a piecewise smooth curve $\sigma\subset f^{-1}(\partial B'_{\varepsilon^L})$ such that
\begin{equation}\label{gh3}
L_{d_t}(\sigma)\le \varepsilon+\hat Ce^{-\frac{1}{2}t}.
\end{equation}
In conclusion, by combining \eqref{gh1}, \eqref{gh2} and \eqref{gh3}, we can choose a positive constant $T$ (independent on $x',x''$) such that for any $t\ge T$,
\begin{equation}
d_t(x',x'')\le(\hat C+2)\varepsilon.
\end{equation}
\par Lemma \ref{gh0} is proved.
\end{proof}

\begin{proof}[Proof of Theorem \ref{main_thm} (2.3)]
Having Lemma \ref{gh0}, we can easily apply arguments in \cite[Section 3]{ZyZz1} (also see \cite[Section 4]{Zy17}) to conclude Theorem \ref{main_thm} (2.3).
\end{proof}

\section{Remarks on the twisted K\"ahler-Ricci flow}\label{sect_rem}
The last section provides more remarks on the twisted K\"ahler-Ricci flow we have studied. Here we remove the conical singularity. Given a projective manifold $X$ with a rational K\"ahler metric $\omega_0\in2\pi c_1(H)$ ($H$ is some ample line bundle on $X$), we have the smooth K\"ahler-Ricci flow starting from $\omega_0$ on $X$:
\begin{equation}\label{KRF}
\left\{
\begin{aligned}
\partial_t\omega(t)&=-Ric(\omega(t))\\
\omega(0)&=\omega_0,
\end{aligned}
\right.
\end{equation}
By \cite{C,Ts,TZo},  the K\"ahler-Ricci flow \eqref{KRF} has a smooth solution up to
$$T_{max}:=\{t>0|[\omega_0]+t2\pi c_1(K_X)>0\}=\{t>0|H+tK_X>0\}.$$
Then $T_{max}=\infty$ if and only if $X$ is a smooth minimal model. 
\par Assume $T_{max}<\infty$. Then $T\in\mathbb Q$ and $H+T_{max}K_X$ is semi-ample (see e.g. \cite{Ma}). By semi-ample fibration theorem there exists a holomorphic map
\begin{equation}\label{map}
f:X\to f(X)\subset\mathbb{CP}^N,
\end{equation}
with the image $Y:=f(X)$ an irreducible normal projective variety of dimension $0\le k\le n$, connected fibers and 
\begin{equation}\label{eqn1}
f^*\mathcal O_{\mathbb{CP}^N}(1)=H+T_{max}K_X.
\end{equation}
Let's first recall some basic ideas of Song and Tian \cite[Section 6.2]{ST17} on contracting or collapsing certain positive part of $c_1(X)$ by the K\"ahler-Ricci flow, which indicates a deep relation between finite-time singularity of the K\"ahler-Ricci flow and the Minimal Model Program in algebraic geometry.
\begin{itemize}
\item [(1)] If $k=0$, then $H=-T_{max}K_X$, i.e. $X$ is a Fano manifold with $\omega_0\in T_{max}c_1(X)$. The K\"ahler-Ricci flow in this case has been studied by many works, see e.g. discussions next to \cite[Conjecture 6.6]{ST17} and references therein for more comments. Here we will focus on the $k>0$ case.
\item [(2)] If $1\le k\le n$, then it is conjectured in \cite{ST17} that, as $t\to T_{max}$, $(X,\omega(t))\to(Y,d_Y)$ in Gromov-Hausdorff topology, here $d_Y$ is some compact metric on $Y$ and should be induced by a smooth K\"ahler metric on $Y$ outside a subvariety.
\end{itemize}
Roughly speaking, \eqref{eqn1} means there are certain "positive parts" in $c_1(X)$ (and so $X$ is not a minimal model), and the picture in the above item (2) means one can use the K\"ahler-Ricci flow to contract/collapse these "positive parts" in Gromov-Hausdorff topology and arrive at a new space, which is closer to a minimal model in some sense (see \cite[Section 6.2]{ST17} for more precise discriptions). There are several progresses on above item (2), see e.g. \cite{Fo,FuZs,SSW,SW,ToZyg1}. However, it seems in general the finite time convergence of the K\"ahler-Ricci flow are still unclear. 
\par La Nave and Tian \cite{LT} proposed a continuity method approach to achieve the above picture, see \cite{LT,LTZ,ZyZz1,ZyZz2,Zy17} for more discussions and results on this direction.\\

\par Here we would like to discuss the possibility of deforming $X$ to $Y$ (given $f:X\to Y$ as in above \eqref{map}) by using a twisted K\"ahler-Ricci flow. Given the above setting, we consider, for an arbitrary K\"ahler metric $\omega_X$,
\begin{equation}\label{TKRF}
\left\{
\begin{aligned}
\partial_t\omega(t)&=-Ric(\omega(t))-\omega(t)-\frac{1}{T_{max}}\omega_0\\
\omega(0)&=\omega_X,
\end{aligned}
\right.
\end{equation}

If we fix a K\"ahler metric $\chi_Y$ on $Y$ with $f^*\chi_Y\in\frac{1}{T_{max}}[\omega_0]+2\pi c_1(K_X)$, then we easily see that the K\"ahler class along \eqref{TKRF} satisfies
$$[\omega(t)]=e^{-t}[\omega_X]+(1-e^{-t})[f^*\chi],$$
which stays positive for any $t\in[0,\infty)$. Therefore, \eqref{TKRF} has a smooth long time solution $\omega(t)$ on $X\times[0,\infty)$. Then one can prove $\omega(t)$ converges to $f^*\overline\chi_Y$ as currents on $X$, where $\overline\chi_Y$ is a K\"ahler current solved by a possibly singular complex Monge-Amp\`ere equation on $Y$. Precisely, we set $V$ be the singular set of $Y$ together with critical values of $f$ and fix a $(1,1)$-form $\overline\omega_X\in[\omega_X]$ on $X\setminus f^{-1}(V)$ such that $\overline\omega_X|_{X_y}$ is a K\"ahler metric on $X_y$ and $Ric(\overline\omega_X|_{X_y})=\frac{1}{T_{max}}\omega_0|_{X_y}$ for all $y\in Y\setminus V$. Also fix a smooth volume form $\Omega$ on $X$ with $\sqrt{-1}\partial\bar\partial\log\Omega=f^*\chi_Y-\frac{1}{T_{max}}\omega_0$. Then $\overline\chi_Y:=\chi_Y+\sqrt{-1}\partial\bar\partial\psi$ is solved by
$$(\chi_Y+\sqrt{-1}\partial\bar\partial\psi)^k=e^{\psi}\frac{\Omega}{C^n_k(\overline\omega_X)^{n-k}\wedge f^*\chi_Y^k}\chi_Y^k.$$
Since $\chi_Y$ is rational and $0<C^{-1}\le\frac{\Omega}{C^n_k(\overline\omega_X)^{n-k}\wedge f^*\chi_Y^k}\in L^{1+\epsilon}(Y,\chi_Y^k)$ (see \cite[Proposition 3.2]{ST12}), we can find a unique $\psi\in PSH(Y,\chi_Y)\cap L^\infty(Y)\cap C^\infty(Y\setminus V)$ solving the above equation (see \cite[Section 3.2, Theorem 3.2]{ST12}).
\begin{itemize}
\item[(a)] Case $0<k<n$. If additionally $Y$ is smooth and $f:X\to Y$ is a submersion, then $\overline\chi_Y$ is a K\"ahler metric on $Y$ and $\omega(t)\to f^*\overline\chi_Y$ in $C^0(X,\omega_0)$-topology exponentially fast (see Theorem \ref{main_thm} (2.2)) and hence $(X,\omega(t))\to(Y,\overline\chi_Y)$ in Gromov-Hausdorff topology. This provides an alternative way (using the twisted K\"ahler-Ricci flow) to deform a Fano bundle to the base in Gromov-Hausdorff topology (compare with the works in \cite{Fo,FuZs,SSW}).
\item[(b)] Case $k=n$. In this case $\overline\chi_Y$ is in fact solved by
$$(f^*\chi_Y+\sqrt{-1}\partial\bar\partial\psi)^n=e^{\psi}\Omega$$
on $X$. Then $\overline\chi_Y$ is a smooth K\"ahler metric on $X\setminus f^{-1}(V)$ and $\omega(t)\to\overline\chi_Y$ smoothly on $X\setminus V$. In fact, it can be checked that $\overline\chi_Y$ coincides with the limit of the continuity method studied in \cite[Theorems 1.1, 1.2]{LTZ}, and so the metric completion of $(X\setminus f^{-1}(V),\overline\chi_Y)$ is a compact metric space homeomorphic to $Y$ by \cite[Theorem 1.2]{LTZ}; denote this limit space by $(Y,d_Y)$. On the other hand, since the twisted term $\frac{1}{T_{max}}\omega_0$ is nonnegative, it is very likely that one can extend a generalized Perelman's no-local-collapsing theorem of Wang \cite[Theorem 1.1]{W} or Q. Zhang \cite[Theorem 6.3.2]{Zq} to our twisted setting. Also note that we have a uniform bound for the twisted scalar curvature $tr_{\omega(t)}(Ric(\omega(t))-\frac{1}{T_{max}}\omega_0)$, see Section \ref{sect_scal}. Therefore, by Wang's argument in \cite[Theorem 8.2]{W}, we may obtain a uniform diameter upper bound along the twisted K\"ahler-Ricci flow \eqref{TKRF} in this volume noncollapsing case (i.e. $k=n$). The remaining question in this case is: \emph{can we prove $(X,\omega(t))\to(Y,d_Y)$ in Gromov-Hausdorff topology?} We will study this question in the future work.
\end{itemize}

\section*{Acknowledgements}
The author thanks Prof. Huai-Dong Cao, Prof. Gang Tian, Prof. Zhenlei Zhang and Dr. Shaochuang Huang for useful discussions, Dr. Jiawei Liu for valuable comments and
 the referee for careful reading and valuable suggestions and comments. 
 Part of this work was carried out while the author was visiting University of Macau and Capital Normal University, which he would like to thank for the hospitality.

\end{document}